\documentclass{amsart}

\usepackage{amsmath, amssymb,amsmath, graphicx, xypic}

\newtheorem{thm}{Theorem}[section]
\newtheorem{lem}[thm]{Lemma}
\newtheorem{cor}[thm]{Corollary}
\newtheorem{prop}[thm]{Proposition}
\theoremstyle{definition}
\newtheorem{defn}[thm]{Definition}

\newtheorem{conjecture}[thm]{Conjecture}

\newtheorem{rem}[thm]{Remark}

\DeclareMathOperator{\area}{\mathsf{area}}
\DeclareMathOperator{\link}{\mathsf{link}}
\DeclareMathOperator{\aristae}{ \mathsf{Sides} }
\DeclareMathOperator{\edges}{ \mathsf{Edges} }

\DeclareMathOperator{\corners}{\mathsf{Corners}}
\DeclareMathOperator{\curvature}{\kappa}
\DeclareMathOperator{\standardcurvature}{\mathsf{Curvature}}
\DeclareMathOperator{\zero}{\mathsf{Zero}}  
\DeclareMathOperator{\bnd}{\mathsf{Bnd}}  
\DeclareMathOperator{\nega}{\mathsf{Neg}} 
\DeclareMathOperator{\isolated}{\mathsf{Isd}}  
\DeclareMathOperator{\positivecells}{\mathsf{Pos}}

\DeclareMathOperator{\negativebound}{\mathsf{N}}  
\DeclareMathOperator{\positivebound}{\mathsf{P}}    
\newcommand{\identity}{\ensuremath{\mathbf{1}}}
\newcommand{\C}{\ensuremath{\mathbb{C}}}
\newcommand{\R}{\ensuremath{\mathbb{R}}}

\newcommand{\mangle}{\measuredangle}
\newcommand{\nclose}[1]{\ensuremath{\langle\!\langle#1\rangle\!\rangle}}
\newcommand{\size}[1]{\ensuremath{\vert #1 \vert}}

\begin{document}

\title[Coherence of Complexes of Groups]{Coherence and Negative Sectional Curvature in Complexes of Groups}
\author[E.Mart\'inez-Pedroza]{Eduardo Mart\'inez-Pedroza}
      \address{Memorial University\\
St. John's, Newfoundland, Canada A1C 5S7}
      \email{emartinezped@mun.ca}

\author[D.~T.~Wise]{Daniel T. Wise}
      \address{
                     McGill University \\
               Montreal, Quebec, Canada H3A 2K6 }
      \email{wise@math.mcgill.ca}
\subjclass[2000]{}
\date{\today}

\begin{abstract}
We examine  a condition on a simply connected $2$-complex $X$ ensuring that groups acting properly on $X$ are coherent. This extends earlier work on 2-complexes with negative sectional curvature~\cite{Wi04} which covers the case that $G$ acts freely. Our extension of these results involves a generalization of the notion of  sectional curvature, an extension of the combinatorial Gauss-Bonnet theorem to complexes of groups, and surprisingly requires the use of $\ell^2$-Betti numbers.  We also prove local quasiconvexity of $G$ under the additional assumption that $X$ is $CAT(0)$ space.
\end{abstract}

\maketitle

\section{Introduction}

A group $G$ is \emph{coherent} if finitely generated subgroups are finitely presented.  A group $G$ is \emph{locally quasiconvex} if each finitely generated subgroup is quasiconvex. A subgroup $H$ of  $G$ is \emph{quasiconvex} if there is a constant $L$, such that every geodesic in the Cayley graph of $G$  that joins two elements of $H$ lies in an $L$-neighborhood of $H$. While $L$ depends upon the choice of Cayley graph, it is well-known that the quasiconvexity of $H$ is independent of the finite generating set when $G$ is hyperbolic. As quasiconvex subgroups are finitely presented, it is clear local quasiconvexity implies coherence.

The class of coherent groups include fundamental groups of compact $3$-manifolds by a result of Scott~\cite{Sc73}, mapping tori of free group automorphisms by work of Feighn and Handel~\cite{FH99}, and one-relator groups with sufficient torsion by McCammond and Wise~\cite{McWi-coherence}.  In contrast, the class of locally quasiconvex groups is substantially smaller. It includes  fundamental groups of infinite volume hyperbolic $3$-manifolds by a result of Thurston (\cite[Prop 7.1]{morgan:uniform} or~\cite[Thm 3.11]{mt:hmkg}),  and there are criteria for local quasiconvexity for certain classes of small cancellation groups~\cite{MaWi10, McWi-coherence}.  

Criteria for proving coherence and local quasiconvexity of groups acting freely on simply connected $2$-complexes was introduced in~\cite{Wi04} based on a notion of combinatorial sectional curvature. These methods do not apply on groups with torsion unless they are known to be virtually torsion free.  It is an  open question whether negatively curved groups are virtually torsion free~\cite{GrHypGps}.

In this paper we revisit the of combinatorial sectional curvature. We provide criteria for coherence and local quasiconvexity of groups acting properly and cocompactly on simply connected $2$-complexes. This extends the methods in~\cite{Wi04} to groups with torsion.   Our extension of these results involves a generalization of the notion of  sectional curvature, an extension of the combinatorial Gauss-Bonnet theorem to complexes of groups,  and surprisingly requires the use of $\ell^2$-Betti numbers. 

We revisit the following notion of sectional curvature in Section~\ref{sec:curvature}.

\begin{defn}[Sectional Curvature $\leq \alpha$]
An {\em angled $2$-complex} $X$ is a combinatorial $2$-complex with an assignment of a real number to each corner of each $2$-cell of $X$. 
A locally finite  angled 2-complex $X$  has {\em sectional curvature at most $\alpha$} if the following two conditions hold:
\begin{enumerate}
\item  for each 0-cell $x$ and each finite subgraph $\Delta$ of $\link (x)$ containing a cycle but no  valence one vertex, we have $\standardcurvature (\Delta) \leq \alpha$ where
\[\standardcurvature (\Delta) = 2\pi -\pi \cdot \chi (\Delta) - \sum_{e\in \edges(\Delta)} \mangle (e),\]
and $\mangle (e)$ is the angle assigned to the corner $e \in \edges(\Delta)$; each edge of the link of a $0$-cell $x\in X$ corresponds to a corner of a $2$-cell whose attaching map contains $x$.

\item  for each 2-cell $f$ of $X$, we have $\standardcurvature (f) \leq 0$, where
\[ \standardcurvature (f) =   \Biggl ( \sum_{c \in \corners(f)} \mangle (c) \Biggr)  -   \pi \Bigl(  |\partial f| - 2   \Bigr) ,\]
where $\corners(f)$ denotes the set of corners of the $2$-cell $f$. 
\end{enumerate}
\end{defn}

\begin{defn}[Angled $G$-complex]
Let $G$ be a group. A complex $X$  equipped with a cellular $G$-action without inversions is a {\em $G$-complex}. A $G$-complex $X$ is {\em proper} (respectively,  {\em cocompact}) if the $G$-action is proper (respectively, cocompact). An {\em angled $G$-complex} is a $2$-dimensional $G$-complex equipped with a $G$-equivariant angle assignment. A $G$-complex is trivial if it is empty or a single point.
\end{defn}

\begin{thm}[Cocompact core] \label{thm:core}
Let $X$ be a simply-connected, proper, and cocompact angled $G$-complex with negative sectional curvature. If $H$ is a subgroup of $G$ and $Y \subseteq X$ is a connected $H$-cocompact subcomplex of $X$, then there is a simply-connected $H$-cocompact subcomplex $Z$ such that $Y \subseteq Z \subseteq X$.
\end{thm}

The strategy of the proof is as follows. A sequence of $H$-equivariant immersions $Y_n\to X$ is constructed inductively from inclusion map $Y_0=Y\to X$. The complex $Y_{i+1}$ is obtained from $Y_i$ by either ``killing a loop" or correcting a ``failure of injectivity." From the construction, a computation shows that the orbifold Euler characteristic of $H \backslash Y_n$ is bounded from below by the first $\ell^2$-Betti number of $H\backslash Y_0$. Using that $X$ has negative sectional curvature, an analysis of the structure of $Y_n$ shows that the number of orbits of $0$-cells with non-negative curvature does not increase with $n$. Then using a version of the combinatorial Gauss-Bonnet theorem for orbihedra and the previous two upper bounds we obtain that the number of orbits of $0$-cells of $Y_n$ with negative curvature is uniformly bounded, and hence the total number of orbits of $0$-cells of $Y_n$ is uniformly bounded. Then a counting argument shows that there are finitely many possibilities for the immersions $Y_i\to X$ up to $G$-equivalence, and therefore the sequence $Y_n\to X$ stabilizes in an embedding $Z\to X$ of a simply connected complex.

\begin{cor}[Coherence criterion]
Let $G$ be a group admitting a proper cocompact action on a simply-connected 2-complex with negative sectional curvature. Then each finitely generated subgroup of $G$ is finitely presented.
\end{cor}
\begin{proof}
Let $X$ a simply connected proper and cocompact $G$-complex with negative sectional curvature. Since $X$ is connected, for any finitely generated subgroup $H\leq G$ there is a connected and cocompact $H$-subcomplex of $X$. By  Theorem~\ref{thm:core} this subcomplex can be assumed to be simply connected.  Then the corollary follows by the well known fact that a group is finitely presented if and only if it acts properly and cocompactly on a simply connected $2$-complex~\cite{M64}. 
\end{proof}

\begin{conjecture}\label{conj:nonpositive}
Let $X$ be a simply-connected, proper, and cocompact angled $G$-complex with sectional curvature $\leq 0$. If $H$ is a subgroup of $G$ and $Y \subseteq X$ is connected $H$-cocompact subcomplex of $X$, then there is simply-connected $H$-cocompact subcomplex $Z$ such that $Y\subseteq Z \subseteq X$.
\end{conjecture}

The main result of the paper is the following criterion for local quasiconvexity.  
A subspace $Y$ of a geodesic space $X$ is {\em quasiconvex} if there is a constant $L$ such that every geodesic in $X$ that joins two elements of $Y$ lies in the $L$-neighborhood of $Y$. 

\begin{thm}[Quasiconvex cores] \label{thm:qccore}
Let $X$ be a $2$-dimensional proper cocompact  $CAT(0)$ $G$-complex whose cells are convex. Assign angles as they arise from the $CAT(0)$-metric. Suppose $X$ has negative sectional curvature.  If $H<G$ and $Y$ is a simply-connected cocompact $H$-subcomplex, then $Y$ is a quasiconvex subspace of $X$. 
\end{thm}

\begin{cor}[Local Quasiconvexity Criterion]
Let $G$ be a group admitting a proper cocompact action on a $2$-dimensional $CAT(0)$-complex with convex cells and negative sectional curvature. Then $G$ is a locally quasiconvex hyperbolic group.
\end{cor}
\begin{proof}
Let $X$ be a $G$-complex as in the statement. 
Since $X$ has negative sectional curvature and angles are positive, $X$ satisfies Gersten's  negative weight test~\cite[Lem. 2.11]{Wi04}. It follows that $X$ satisfies a linear isoperimetric inequality and hence $X$ is a $\delta$-hyperbolic space and $G$ is a hyperbolic group~\cite[Thm. A6]{Ge87}.  By Theorem~\ref{thm:core}, every finitely generated subgroup $H$ of $G$ admits a simply-connected and $H$-cocompact subcomplex of $X$; then Theorem~\ref{thm:qccore} implies that these subcomplexes are quasiconvex in $X$. 
\end{proof}

The strategy of the proof of Theorem~\ref{thm:qccore} is the following. Let $\ell$ be a geodesic with respect to the $CAT(0)$-metric in $X$ with endpoints in the $0$-skeleton of $Y$. Since $\ell$ is not a combinatorial path in the cell structure of $X$, we approximate $\ell$ with a combinatorial path $P_L\to X$ which has the same endpoints and is uniformly close. We also take a path $P_Y \to Y$ between the endpoints of $\ell$. The choices of $P_Y$ and $P_L$ are made under some technical assumptions, in particular, they minimize the area of the disk-diagram $D$ with boundary cycle $P_L P_Y^{-1}\to X$. Let $u$ be a $0$-cell of $P_L\to X$. Analyzing the cell-structure of $D$, we show that there exists a {\em good path} $Q\to D$ with initial point $u$ and terminal point in $P_Y\to X$; by good we mean that if a $0$-cell $x$ of $Q\to D\to X$ intersects $\ell$ then $x$ is in the interior of $D$.  Using that $X$ has a $CAT(0)$-structure,  we construct an $H$-equivariant immersion $Z'\to X'$ where $X'$ is a subdivision of $X$, any good path $Q\to D\to X$ lifts as an internal path of $Q'\to Z'$ (after subdividing), and the number of $H$-orbits of $0$-cells with negative curvature of $Z'$ is bounded by a constant independent of $\ell$. By the existence of good paths, we can take a path $Q\to X$ of minimal length from $u$ to $P_Y\to X$ such that there is a lifting $Q'\to Z'$ which is internal. Then the proof concludes in the following way. Since $X'$ has non-positive sectional curvature, if a $0$-cell $x$ of $Z'$ is internal i.e., $\link(x)$ has a cycle, then its curvature is non-positive. Since $X$ has negative sectional curvature, if $x$ has zero curvature in $Z'$ and is internal then its image in $X$ is not a $0$-cell. Therefore the length of $Q\to X$ equals the number of $0$-cells of $Q'\to Z'$ with negative curvature plus $1$; by minimality no two $0$-cells of $Q\to X$ are in the same $H$-orbit; therefore $|Q|$ is bounded by the number of orbits of  $0$-cells of $Z'$ with negative curvature plus $1$. By construction of $Z'$ this number is uniformly bounded independently of $\ell$.

A $(p,q,r)$-complex is a combinatorial 2-complex $X$ such that the attaching map of each $2$-cell has length $\geq p$, for each $x\in X^0$ the $\link (x)$ has girth $\geq q$, and each $1$-cell $e$ of $X$ appears $\leq r$ times among the attaching maps of $2$-cells. The second author provide a criterion for negative sectional curvature of $(p,q,r)$-complexes in~\cite{wisepqr} from which the following application follows. 

\begin{cor}
Let $X$ be a $2$-complex. Suppose that
\begin{enumerate}
\item $X$ is $(p, 3, p-3)$-complex for $p \geq 7$,
\item $X$ is $(p, 4, p-2)$-complex for $p \geq 5$, or
\item $X$ is a $(p, 5, p-1)$-complex where $p \geq 4$.
\end{enumerate}
Then any group $G$ acting properly and cocompactly on $X$ is a locally quasiconvex hyperbolic group.
\end{cor}

\subsection*{Outline of the paper.} Section~\ref{sec:preliminaries} discusses preliminaries. Section~\ref{sec:curvature} contains definitions of sectional curvature and generalized sectional curvature which are used during the rest of the paper. Section~\ref{sec:gaussbonnet} discusses the combinatorial Gauss-Bonnet theorem for angled $G$-complexes.  Section~\ref{sec:betti} recalls some results in the literature on $\ell^2$-Betti numbers. 
Section~\ref{sec:folding} discusses equivariant immersions as a preliminary of the proof of the  main results of the paper. The  proof of the simply connected core theorem is the content of Section~\ref{sec:core}.  Section~\ref{sec:qccore} contains the proof of  Theorem~\ref{thm:qccore}. The last section discusses a criterion establishing that certain quotiens of locally quasiconvex groups are  locally quasiconvex.

\subsection*{Acknowledgment:}
The authors thanks the referee for excellent feedback, and for pointing out a serious gap in an earlier version of the paper.  Both authors are supported by NSERC.

\section{Preliminaries}\label{sec:preliminaries}

\subsection{Complexes and Disk diagrams}

This paper follows the notation used in \cite{McWi-fans}, and in this section we quote various of those relevant notations for the convenience of the reader. All complexes considered in this paper are combinatorial $2$-dimensional complexes, and all maps are combinatorial.

\begin{defn}[Path and Cycle]\cite[Def 2.5]{McWi-fans}
A \emph{path} is a map $P\rightarrow X$ where $P$ is a subdivided
interval or a single $0$-cell.  
A \emph{cycle} is a map $C\rightarrow X$ where $C$ is a subdivided circle.  
Given two paths $P\rightarrow X$ and $Q\rightarrow X$ such that the terminal point of $P$ and
the initial point of $Q$ map to the same $0$-cell of $X$, their
concatenation $PQ\rightarrow X$ is the obvious path whose domain
is the union of $P$ and $Q$ along these points.  The path
$P\rightarrow X$ is \emph{closed} if the
endpoints of $P$ map to the same $0$-cell of $X$.  A path or cycle
is \emph{simple} if the map is injective on $0$-cells.
The \emph{length} of the path $P$ or cycle $C$ is the number of
$1$-cells in the domain and is denoted by $\size{P}$ or
$\size{C}$.  The \emph{interior of a path} is the path minus its
endpoints.   A \emph{subpath} $Q$ of a path $P$ is given by a path $Q \rightarrow P
\rightarrow X$ in which distinct $1$-cells of $Q$ are sent to distinct $1$-cells of $P$.
\end{defn}

\begin{defn}[Disc Diagram]\cite[Def 2.6]{McWi-fans}\label{def:diagrams}
A {\em disc diagram} $D$ is a compact contractible $2$-complex with a
fixed embedding in the plane.  A {\em boundary cycle} $P$ of $D$ is a
closed path in $\partial D$ which travels entirely around $D$ (in a
manner respecting the planar embedding of $D$). For a precise definition we refer the reader to~\cite{McWi-coherence}.

Let $P\rightarrow X$ be a closed null-homotopic path.  A {\em disc
diagram in $X$ for $P$} is a disc diagram $D$ together with a map
$D\rightarrow X$ such that the closed path $P\rightarrow X$ factors as
$P\rightarrow D\rightarrow X$ where $P\rightarrow D$ is the boundary
cycle of $D$. Define $\area (D)$ as the number of $2$-cells in $D$.
\end{defn}

\begin{defn}[Arc]\cite[Def 5.4]{McWi-fans}
An \emph{arc} in a diagram $D$ is an embedded path $P\rightarrow D$  such that each of its internal 0-cells is mapped to a 0-cell with valence 2 in $D$.  The arc is \emph{internal} if its interior lies in the interior of $D$, and it is a \emph{boundary arc} if it lies entirely in $\partial D$.
\end{defn}

\begin{defn}[Internal path]\label{defn:internal}
A path $P\to X$ is {\em internal} if each $0$-cell in the interior of $P$ is mapped to a $0$-cell of $X$ whose link contains an embedded cycle.
\end{defn}

\begin{defn}[Links]\cite[Def 4.1]{McWi-fans}
Let $X$ be a locally finite complex and let $x$ be a $0$-cell of $X$.
The cells of $X$ each have a natural partial metric obtained by making every $1$-cell isometric to the unit interval and every $n$-sided $2$-cell isometric to a Euclidean disc of circumference $n$
whose boundary has been subdivided into $n$ curves of length $1$. In this
metric, the set of points which are a distance equal to $\epsilon$ from $x$ will form a
finite graph. If $\epsilon$ is sufficiently small, then the graph obtained is independent
of the choice of $\epsilon$. This well-defined graph is the link of $x$ in $X$ and is denoted by $\link (x, X)$. 
\end{defn}

\begin{defn}[Immersions, near-immersions]\cite[Def 2.13]{McWi-fans}
The map $Y\to X$ is an {\em immersion} if it is locally injective. 
The map $Y \to X$ is a {\em near-immersion} if $Y \backslash Y^{(0)} \to X$ is
locally injective. Equivalently, a map is an immersion if the induced maps on links of $0$-cells are embeddings and a map is a near-immersion if the induced maps on links of $0$-cells are immersions.
\end{defn}
The following lemma is essential for the rest of the paper. It is an immediate consequence of the fact that  immersions of graphs are $\pi_1$-injective.
\begin{lem}[Near-immersions map internal $0$-cells to internal $0$-cells]
Let $X\to Y$ be a near-immersion mapping the $0$-cell $x$ to $y$. If $\link (x)$ has an embedded cycle then $\link (y)$ has an embedded cycle.
\end{lem}

\begin{defn} [Corners]\cite[Def 4.2]{McWi-fans} Let $X$ be a $2$-complex, let $x$ be a $0$-cell of $X$,
and let $R\to X$  be a $2$-cell of $X$. Regard the $2$-cells of $X$ as polygons. Then the edges of
$\link (x)$ correspond to the corners of these polygons attached to $x$. We will
refer to a particular edge in $\link (x)$ as a corner of $R$ at $x$.
\end{defn}

\subsection{$G$-complexes }

All group actions on complexes are without inversions, i.e., a setwise stabilizer of a cell is a pointwise stabilizer. Under this assumption, quotients of complexes by group actions have an induced cell structure. 

\begin{defn}[$I^p(G, X)$ and $|G_\sigma|^{-1}$]
Let $X$ be a  $G$-complex $X$. Let $I^p(G, X)$ be the set of orbits of $p$-dimensional cells.  For $\sigma \in I^p(G, X)$, let $|G_\sigma|^{-1}$ denote the reciprocal of the order of the $G$-stabilizer of a representative of $\sigma$ in $X$, where $|G_\sigma|^{-1}$ is understood to be zero if the order is infinite.
\end{defn}

\begin{defn}[Angled $K$-graph]
Let $K$ be a group. An {\em angled $K$-graph} is a graph $\Gamma$ equipped with a $K$-action and a $K$-map $\mangle: \edges (\Gamma) \to \R$.  For an edge $e$ of $\Gamma$, the number $\mangle (e)$ is called the {\em angle at $e$}.
\end{defn}

\begin{rem}[Connection with angled $G$-complexes]
If $X$ is an angled $G$-complex and $x$ is a 0-cell of $X$, each edge $e$ of $\link (x)$ corresponds to a corner of a 2-cell of $X$ and is thus associated to a real number $\mangle (e)$. This assignment of angles to the edges of $\link (x)$ is preserved under the $G_x$-action. In particular, $\link (x, X)$ is an angled $G_x$-graph.
\end{rem}

\section{Curvature}~\label{sec:curvature}

\begin{defn}[Curvature of $K$-graphs]\label{defn:weight}
Let $K$ be a group. For a cocompact angled $K$-graph $\Gamma$, define
\begin{equation*}
\standardcurvature (K, \Gamma) = 2 \pi \cdot |K|^{-1}  -  \sum_{v \in I^0 (K, \Gamma) }\pi \cdot |K_v|^{-1} + \sum_{e \in I^1 (K, \Gamma)} \Bigl(\pi- \mangle (e) \Bigr) \cdot |K_e|^{-1}.
\end{equation*}
\end{defn}

\begin{defn}[Regular Section] \label{defn:section}
Suppose that $\Gamma$ is an angled $K$-graph and $H$ is a subgroup of $K$. An {\em $H$-subgraph} $\Delta$ is an $H$-invariant subgraph of $\Gamma$, and an {\em $H$-section} is an $H$-subgraph which is $H$-cocompact. An edge having a vertex with valence one is called a {\em spur}, and a graph with no vertices with valence one is called {\em spurless}.  A spurless, connected, and not edgeless $H$-section is called {\em regular}.  An edgeless $H$-section is called {\em trivial}. 
\end{defn}

\begin{defn}[Sectional Curvature $\leq \alpha$] \label{defn:sec-curvature}\label{defn:seccurv}
An angled complex $X$  has {\em sectional curvature $\leq \alpha$} if the following two conditions hold.
\begin{enumerate}
\item  for each 0-cell $x$,  each regular section of the angled $\identity$-graph $\link (x)$ has curvature $\leq \alpha$, where  $\identity$ denotes the trivial group.
\item for each 2-cell $f$ of $X$, we have $\standardcurvature (f) \leq 0$, where
\[ \standardcurvature (f) =   \Biggl ( \sum_{c \in \corners(f)} \mangle (c) \Biggr)  -  \pi \Bigl(  |\partial f| - 2   \Bigr) .\]
\end{enumerate}
If $\alpha\leq 0$ then we say that $X$ has {\em nonpositive sectional curvature}. 
\end{defn}

\begin{defn}[Generalized Sectional Curvature $\leq \alpha$] \label{defn:sec-curvature}
An  angled $G$-complex $X$  has {\em generalized sectional curvature $\leq \alpha$} if:
\begin{enumerate}
\item  for each 0-cell $x$, and each $H\leq G_x$,  each regular $H$-section of the angled $G_x$-graph $\link (x)$ has curvature $\leq \alpha$, and
\item for each 2-cell $f$ of $X$, we have $\standardcurvature (f) \leq 0$.
\end{enumerate}
\end{defn}

In the case that $X$ is a proper $G$-complex, the two notions above are equivalent.  
\begin{prop}\label{prop:standard-generalized}
Let $X$ be a proper, cocompact and angled $G$-complex with sectional curvature $\leq \alpha \leq 0$. 
Let $K=K(G, X)$ an upper bound for the cardinality of $0$-cell stabilizers.  Then $X$ has generalized sectional curvature $\leq \alpha/K$. 
\end{prop}
\begin{proof}
Let $x$ be a $0$-cell of $X$, let $H$ be a subgroup of $G_x$, and let $\Delta$ be an $H$-invariant regular section of $\link (x)$. The result follows by observing that 
\[ |H| \cdot \standardcurvature (H, \Delta) =   \standardcurvature (\identity, \Delta). \qedhere \]
\end{proof}

\begin{defn}[$\corners$, $\aristae$]
Let $X$ be a $G$-complex and suppose that $G$ acts without inversions on $X$. For $v \in I^0(G, X)$, let $\corners (v)$ and $\aristae (v)$ denote the sets of edges and vertices of the link of $v$ in $G\backslash X$.   Let $\corners (G,  X)$ denote the disjoint union $\bigcup_{v\in I^0} \corners (v)$ and analogously for $f \in I^2(G, X)$, let $\corners (f)$ denote the subset of $\corners (G, X)$ determined by $f$.  For $e \in \corners (v)$,  let $|G_e|^{-1}$ denote $|G_\sigma|^{-1}$ where $\sigma$ is the 2-cell of $G\backslash X$ determined by $e$. For $a \in \aristae (v)$ define $|G_a|^{-1}$ analogously.
\end{defn}

\begin{rem}\label{rem:2to1}
Each element of $\aristae (v)$ is determined by a $1$-cell in $G\backslash X$. In particular, there is a natural two-to-one surjection
\begin{equation}\nonumber \bigcup_{v \in I^0} \aristae (v) \longrightarrow I^1(G, X).\end{equation} 
\end{rem}

\begin{defn} 
Let $X$ be a cocompact $G$-complex.  For $v \in I^0(G, X)$, the curvature $ \curvature
 (v)$ is defined by:
\begin{equation*}
\begin{split}
  \curvature (v)  & =     2\pi \cdot |G_v|^{-1} \  -  \sum_{e \in \aristae (v)} \pi \cdot |G_e|^{-1} \  +   \sum_{c \in \corners(v)} \Bigl(\pi - \mangle (c) \Bigr)\cdot |G_c|^{-1}.
\end{split}
\end{equation*}
The curvature of $f \in I^2(G, X)$ is defined by:
\begin{equation*}
\begin{split}
 \curvature ( f ) & =   \left [   \Biggl ( \sum_{c \in \corners(f)} \mangle (c) \Biggr)  -   \pi \Bigl(  |\partial f| - 2   \Bigr)   \right] \cdot  |G_f|^{-1}.\\
\end{split}
\end{equation*}
\end{defn}


\begin{rem}
Let $X$ be a cocompact angled $G$-complex, let $v \in I^0(G, X)$ and let $f \in I^2(G, X)$. 
Observe that $\curvature (v)$ and $\curvature (f)$ are finite real numbers. Moreover, 
\[\curvature (v) = \standardcurvature (G_x , \link (x) )\] where $x$ is a representative of $v$ in $X$, and analogously, \[ \curvature (f) = \standardcurvature (\sigma) \cdot |G_\sigma|^{-1}\] where $\sigma$ is a representative of $f$ in $X$.
\end{rem}

\section{The Combinatorial Gauss-Bonnet Formula}\label{sec:gaussbonnet}

\begin{defn}[Euler Characteristic]\label{defn:Euler-ch}
Let $X$ be a $2$-dimensional cocompact $G$-complex. The {\em Euler characteristic $\chi (G, X)$} is defined by: 
\[ \chi (G, X) = \sum_{\sigma \in I^0(G, X)}    |G_\sigma|^{-1}-\sum_{\sigma \in I^1(G, X)}  |G_\sigma|^{-1}+\sum_{\sigma \in I^2(G, X)}   |G_\sigma|^{-1}.\]
\end{defn}

\begin{thm}[Combinatorial Gauss-Bonnet] \label{thm:Gauss-Bonnet}
If $X$ is an angled and cocompact $G$-complex, then
\begin{equation}\label{eq:gauss-bonnet}
2\pi  \cdot \chi (G, X) =  \sum_{v \in I^0(G, X)}  \curvature (v) \ + \sum_{f \in I^2(G, X)}  \curvature (f).
\end{equation}
\end{thm}
\begin{proof}
From the definition of $\curvature (v)$ and the natural two-to-one surjection of Remark~\ref{rem:2to1}, we have
\begin{equation}\label{eq:0-cell-curvatures}
\sum_{v \in I^0} \curvature (v)\ =\  \sum_{v \in I^0} 2\pi \cdot |G_v|^{-1} - \sum_{e \in I^1} 2\pi \cdot |G_e|^{-1} +   \sum_{v \in I^0} \sum_{c \in \corners(v)} 
\Bigl(\pi- \mangle (c) \Bigr)\cdot |G_c|^{-1}.
\end{equation}

Observe that
\begin{equation}\label{eq:chain-1}
\begin{split}
\sum_{ f\in I^2 } \curvature (f)  = &  
 \sum_{ f \in I^2 } 2\pi \cdot |G_f|^{-1} \  + \   \sum_{ f \in I^2 } \left [  \Biggl (  \sum_{c \in \corners(f)}   \mangle (c)  \Biggr)   -   \pi \cdot |\partial f| \right ] \cdot |G_f|^{-1}  \\
 & =
  \sum_{ f \in I^2 } 2\pi \cdot |G_f|^{-1} \   - \   \sum_{ f \in I^2 } \sum_{c \in \corners(f)}   \Bigl(\pi- \mangle (c) \Bigr)\cdot |G_f|^{-1},
\end{split}
\end{equation}
where the first equality follows from the definition of $\curvature (f)$, and the second equality holds since $|\partial f| = |\corners(f)|$ for each 2-cell $f$.

Moreover, 
\begin{equation}\label{eq:chain-2}
\begin{split}
\sum_{f \in I^2}   \sum_{c \in \corners(f)}  \Bigl(\pi- \mangle (c) \Bigr)\cdot |G_f|^{-1} 
&=
\sum_{f \in I^2}   \sum_{c \in \corners(f)}  \Bigl(\pi- \mangle (c) \Bigr)\cdot |G_c|^{-1} \\
& = 
\sum_{v \in I^0} \sum_{c \in \corners(v)} \Bigl(\pi- \mangle (c) \Bigr)\cdot |G_c|^{-1},
\end{split}
\end{equation}
where the first equality follows from $G_c = G_f$ for each $c \in \corners (f)$, and the second equality holds since $\{\corners (v) \}_{v\in I^0}$ and $\{\corners (f) \}_{f\in I^2}$ are both partitions of the set $\corners (G, X)$.

The chains of equalities~\eqref{eq:chain-1} and~\eqref{eq:chain-2} imply that
\begin{equation}\label{eq:2-cell-curvatures}
 \sum_{ f\in I^2 } \curvature (f)\ =\ \sum_{ f \in I^2 } 2\pi \cdot |G_f|^{-1}  -   \sum_{ v \in I^0 } \sum_{c \in \corners(v)} 
\Bigl(\pi- \mangle (c) \Bigr)\cdot |G_c|^{-1} .
\end{equation}

The Gauss-Bonnet formula~\eqref{eq:gauss-bonnet}  follows by adding equations~\eqref{eq:0-cell-curvatures} and~\eqref{eq:2-cell-curvatures}.
\end{proof}

\section{Euler Characteristic and $\ell^2$ Betti numbers}\label{sec:betti}

For a  $G$-complex $X$, the $p$-th $\ell^2$ Betti number $b_p^{(2)}(G, X)$ of $X$ is a element of the extended interval $[0, \infty]$. We follow the approach by W.L\"uck and we refer the reader to~\cite{Lu02} for definitions and a general exposition on the subject. The approach by  L\"uck to $\ell^2$-Betti numbers  fits the work of this paper since there are no assumptions on the $G$-action on $X$; in particular, the $G$-action is not required to be free.

\begin{thm}[Atiyah's Formula\hbox{~\cite[Thm 6.80]{Lu02}}]\label{thm:lu-2} 
For a cocompact $G$-complex $X$,  \begin{equation*}\chi (G, X)  = \sum_{p \geq 0}  (-1)^{p} \cdot b_p^{(2)}(G, X). \end{equation*}
\end{thm}

\begin{defn}For a $G$-space $X$ and $H\leq G$, let $X^H$ denote the subspace of $X$ consisting of points fixed by all elements of $H$. For a $G$-map $X\to Y$ and $H<G$,  $f(X^H) \subseteq Y^H$; denote by $f^H$ the restriction of $f$ to $X^H \to Y^H$.
\end{defn}

\begin{thm}\cite[Thm~6.54(1)]{Lu02} \label{thm:lu-1}
Let $X$ and $Y$ be $G$-complexes, and let $f\colon X \to Y$ be a $G$-map.  Suppose for $n\geq 1$ that for each subgroup $H\leq G$ the induced map $f^H: X^H\rightarrow Y^H$ is $\C$-homologically $n$-connected, i.e., 
the map
\[ H_p^{sing}(f^H; \C)\colon H_p^{sing}(X^H; \C) \longrightarrow  H_p^{sing}(Y^H; \C) \]
induced by $f^H$ on singular homology with complex coefficients is bijective for $p<n$ and surjective for $p=n$. Then 
\[ b_p^{(2)}(G, X) = b_p^{(2)}(G, Y) \  \   \   \   \text{for $p<n$;} \]
\[ b_n^{(2)}(G, X) \geq b_n^{(2)}(G, Y) \  \   \   \   \text{for $p=n$}. \]
\end{thm}

\begin{thm}\cite[Thm~6.54(3)]{Lu02} \label{thm:lu-3}
Let $X$ be a $G$-complex. Suppose that for all $x \in X$ the stabilizer $G_x$ is finite or satisfies $b_p^{(2)}(G_x)=0$ for all $p\geq 0$. Then 
$b_p^{(2)}(G, X)=b_p^{(2)}(G, EG\times X)$ for $p\geq 0$ where $EG$ is a classifying space for $G$. 
\end{thm}

\begin{cor}\label{cor:betti-goal}
Let $X$ and $Y$ be connected $G$-complexes, and let 
$f\colon X \to Y$ be a $G$-map such that the induced map on singular homology $f_* \colon H_1(X, \C) \to H_1 (Y, \C)$ is surjective.  Suppose that for all $x \in X$ the stabilizer $G_x$ is finite or satisfies $b_p^{(2)}(G_x)=0$ for all $p\geq 0$, and analogously for all $y\in Y$.  Then $b_1^{(2)}(G, X) \geq b_1^{(2)}(G, Y)$.
\end{cor}
\begin{proof}
Since $EG\times X$ and $EG\times Y$ are connected, free $G$-complexes, the map $\identity \times f \colon EG\times X \to EG\times Y$ is
$\C$-homologically $1$-connected. By Theorem~\ref{thm:lu-1}, it follows that $b_1^{(2)}(G, EG\times X)  \geq b_1^{(2)}(G, EG\times Y)$.
Since isotropy groups are finite or satisfy $b_p^{(2)}(G_x)=0$ for $p\geq 0$,  Theorem~\ref{thm:lu-3} implies that $ b_1^{(2)}(G, X) \geq  b_1^{(2)}(G, Y)$. 
\end{proof}

\begin{cor}\label{cor:betti-goal-2}
Let $X$ and $Y$ be contractible $G$-complexes, let 
$f\colon X \to Y$ be a $G$-map, and suppose that for all $x \in X$ the stabilizer $G_x$ is finite or satisfies $b_p^{(2)}(G_x)=0$ for all $p\geq 0$, and analogously for all $y\in Y$.  Then $b_p^{(2)}(G, X) = b_p^{(2)}(G, Y)$ for $p\geq 0$. In particular, $\chi (G, X) = \chi (G, Y)$.
\end{cor}
\begin{proof}
Since $EG\times X$ and $EG\times Y$ are free $G$-complexes, 
the K\"unneth formula implies that the map $\identity \times f \colon EG\times X \to EG\times Y$ is $\C$-homologically $n$-connected for all $n$. Theorem~\ref{thm:lu-1} implies $b_p^{(2)}(G, EG\times X)  = b_p^{(2)}(G, EG\times Y)$ for $p\geq 0$.
Since isotropy groups are finite or satisfy $b_p^{(2)}(G_x)=0$ for $p\geq 0$,  Theorem~\ref{thm:lu-3} implies that $ b_p^{(2)}(G, X) =  b_p^{(2)}(G, Y)$ for $p\geq 0$. 
\end{proof}

\section{Equivariant Immersions} \label{sec:folding}

\begin{defn}
Let $X$ be a $G$-complex. Define \begin{itemize}
\item $\bnd (G, X)$ as the subset of $v\in I^0(G, X)$ such that $\link (x, X)$ is a single vertex or has spurs for a representative $x \in X$ of $v$, and
\item $\isolated (G, X)$ as the subset of $v\in I^0(G, X)$ such that $\link (x, X)$ has at least one vertex of valence zero for a representative $x \in X$ of $v$.
\end{itemize}
\end{defn}

\begin{defn}[Essential Path]
A path $P\to Y$ is {\em essential} if the  lift to the universal cover $P\to \widetilde Y$ is not closed.
\end{defn}

\begin{thm}[Collapsing Essential Paths]\label{thm:collapsing-essential-paths}
Let $X$ be a simply-connected $H$-complex, let $Y\to X$ be an $H$-equivariant immersion, and let $P\to Y$ be an essential path.  Suppose that $P\to Y \to X$ is a closed path that is simple in the sense that it embeds except at its endpoints.  Then $Y\to X$ factors as a composition $Y\to Z \to X$ of $H$-equivariant maps where $Y\to Z$ is $\pi_1$-surjective, $Z\to X$ is an immersion, and the path $P\to Y\to Z$ is closed and null-homotopic. Moreover we can choose $Z$ such that the following hold:
\begin{enumerate}
\item \label{thm:collapsing-1}
 If $Y$ is connected then $Z$ is connected. 
\item \label{thm:collapsing-2} 
If $Y$ is $H$-cocompact then $Z$ is $H$-cocompact.
\item \label{thm:collapsing-3} 
$|\bnd (H, Z) \cup \isolated (H, Z)| \leq |\bnd (H, Y) \cup \isolated (H, Y)|$.
\end{enumerate}
\end{thm}

The strategy of the proof is as follows: A disk diagram $D\to X$ with boundary path $P\to X$ is equivariantly attached to $Y$ to obtain an $H$-complex $Z'$ and $H$-maps $Y\to Z'\to X$; this is performed using pushouts of equivariant immersions. Then the complex $Z'$ is equivariantly folded to obtain $H$-maps $Z'\to Z\to X$ such that $Z\to X$ is an immersion. The proof of the theorem requires some preliminary results.

\subsection{Equivariant Folding and Pushouts}

\begin{lem}[Equivariant Folding]\label{lem:folding}
Let $W\to X$ be a $G$-map with $W$ locally finite. Then $W\to X$ factors as $W\to Z  \to X$ where $Z \to X$  is an immersion, $W \to Z $ is a surjection and a $\pi_1$-surjection, and all maps are $G$-invariant.
\end{lem}
\begin{proof}
The statement is well-known when $W$ is compact and $G$ is trivial. In general, let $W= \bigcup_i W_i$ be a filtration by compact sets.  For each $i$, let $W_i\to Z_i \to X$ be a factorization such that $Z_i\to X$ is an immersion and $W_i\to Z_i$ is surjective and $\pi_1$-surjective. Observe that for $i<j$ there is a commutative diagram on the left below.

\[\begin{array}{ccc} 
   W_i & \subseteq & W_j \\
\downarrow &  & \downarrow \\
   Z_i & \to & Z_j
 \end{array} \hspace{3cm}  
\begin{array}{cccccc}
  U_1 & \to & U_{2} 	& \to 	& U_{3}  & \\
  \downarrow &  		& \downarrow &  & \downarrow & \cdots \\
  V_1 & \to & V_{2} & \to & V_{3} & \\
\end{array}\]

For a sequence $U_1\to U_2 \to U_3 \to \cdots$, we define $U_\infty$ to be the direct limit. Specifically, $U_\infty$ is the combinatorial complex whose $p$-cells are tales of $p$-cells $c_i\to c_{i+1} \to c_{i+2} \to \cdots$ for some $i\geq 1$, where  two tales are equivalent if they are eventually the same.  If $\{U_i\}$ is a filtration of $U$, then $U_\infty$ equals $U$.  A morphism $\{U_i\} \to \{V_i \}$  between two such sequences is a commutative diagram as on the right above. Any such a morphism induces a map $U_\infty \to V_\infty$. Observe that if the vertical arrows of the morphism are surjective, then the map $U_\infty \to V_\infty$ is surjective. 

By surjectivity, the morphism $\{W_i\} \to \{Z_i\}$ induces a surjective map $W_\infty \to Z_\infty$. 
As above, $W_\infty$ equals $W$. Let $X_i=X$ and let $X_1\to X_2 \to X_3 \to \cdots$ be the identity sequence, and note that $X_\infty$ equals $X$.  It follows that we have a maps $W\to Z \to X$ where $Z=Z_\infty$. 

Let us verify that $Z\to X$ is an immersion. Since each $Z_i \to X$ is an immersion and $X$ is locally finite, $Z$ is locally finite. In particular, for each $0$-cell $z \in Z$ there is an index $i$ for which $Z_i \to Z$ maps $z_i \mapsto z$ and $\link (z_i)$ maps isomorphically onto $\link (z)$; the factorization $Z_i\to Z\to X$ shows that $Z\to X$ is locally injective at $z$. 

To see that $Z$ is an $H$-complex and $W\to Z$ is an $H$-map, we observe that for $h\in H$, there is another filtration $W=\bigcup_i hW_i$.  For each $i$ there is $i_*$ such that $hW_i \subseteq W_{i_*}$ and this induces a map $h_i \colon Z_i \to Z_{i_*}$; it follows that $\{h_i\}$ induces a map $h:Z\to Z$ defining an action of $H$ onto $Z$. By construction, the action commutes with the map $W \to Z$. 

To see that $W\to Z$ is $\pi_1$-surjective, observe that each closed path $\sigma$ in $Z$ occurs in some $Z_i$, and is thus homotopic to (the image of a) closed path $\sigma_i\to W_i\to W$
\end{proof}

\begin{lem}[Pushouts of Equivariant Maps]\label{lem:pushout}
Let $\phi\colon C\to A$ and $\psi\colon C\to B$ be $G$-maps of complexes. Then there is a $G$-complex $Z$ and $G$-maps $\imath\colon A\to Z$ and $\jmath\colon B \to Z$ such that $\imath\circ \phi = \jmath \circ \psi $. The pushout $(Z, \imath, \jmath)$ is universal in the sense that equivariant maps $A\to X$ and $B\to X$  for which the diagram below commutes induce a unique equivariant map $Z\to X$ also making the diagram commute. 
\begin{equation*}
 \xymatrix{
			   								  & A  \ar[rd]_{\imath}	\ar[rrrd]			    &  	 				&&     	\\
C \ar[rd]_\psi \ar[ru]^{\phi}  &  				  			    &   Z \ar@{-->}[rr] 		&& X	\\
			  			  					 & B \ar[ru]^{\jmath} \ar[rrru] 	&    	&&   \\
}
 \end{equation*}
 Moreover, if $A$ and $B$ are $G$-cocompact (proper) then $Z$ is $G$-cocompact (respectively proper). If $A$ is connected and $\psi (C)$ intersects all connected components of $B$  then $Z$ is connected.
\end{lem}
\begin{proof}
The construction of  the pushout of $\phi$ and $\psi$ is standard and it is briefly described. Let $Z$ be the combinatorial complex obtained by taking the quotient of the disjoint union of $A$ and $B$ by the relation $\phi(\sigma)=\psi(\sigma)$ for $\sigma$ in $C$. The resulting complex admits a natural $G$-action.  The statements on inheritance of cocompactness, properness and connectedness are routine and details are left to the reader. 
\end{proof}

\subsection{Proof of Theorem~\ref{thm:collapsing-essential-paths}}

Since $X$ is simply-connected, there is a near-immersion $D\to X$ of a disk diagram with boundary path $P\to X$. Since  $P\to X$ is a simple closed path, $D$ is homeomorphic to an Euclidean disk. 
 
Let $\bigcup_H P$ denote the disjoint union of copies of $P$, one for each element of $H$. Define  $\bigcup_HD$ analogously. By Lemma~\ref{lem:pushout}, let $Z'$ be the pushout of the natural $H$-maps $\bigcup_H P \to \bigcup_H D$ and $\bigcup_H P \to Y$. By the universal property of the pushout, the immersion $Y\to X$ and the natural map $\bigcup_H D \to X$ induce an $H$-map $Z'\to X$. By Lemma~\ref{lem:folding}, let $Z\to X$ be the $H$-equivariant immersion obtained after an equivariant folding of $Z'\to X$. 
 We refer to the following  commutative diagram.
\begin{equation*}
\xymatrix{
			   			  & Y  \ar[rd] \ar[rrrd]    &  	 &    & 	\\
\bigcup_H P \ar[rd] \ar[ru]  &  				  			    &   Z' \ar[r] & Z\ar[r] &  X  	\\
			  			  & \bigcup_H  D \ar[ru] \ar[rrru] 	&      	&    &
}
\end{equation*}

The main conclusions of Theorem~\ref{thm:collapsing-essential-paths} are proved in the four lemmas below.

\begin{lem}
$Y\to Z'\to Z$ is $\pi_1$-surjective.
\end{lem}

\begin{proof}
By Lemma~\ref{lem:folding},  $Z'\to Z$ is $\pi_1$-surjective. 
Hence it remains to prove that $Y\to Z'$ is $\pi_1$-surjective.  
Let $Y'$ denote the image of $Y$ in $Z'$, and note that $Y'\to Z'$ is $\pi_1$-surjective since the complement of $Y'$ in $Z'$ is a collection of disjoint open disks.
Observe that $\pi_1 Y'$ is generated by the image of $\pi_1 Y \to \pi_1 Y'$ together with closed paths corresponding to $H$-translates of $P \to Y\to Z'$. 
Since these additional paths become null-homotopic by the addition of the $H$-translates of $D\to Z'$, the result follows. 
\end{proof}

\begin{lem}
If $Y$ is $H$-cocompact then $Z$ is $H$-cocompact.  If $Y$ is connected then $Z$ is connected. 
\end{lem}
\begin{proof}
Suppose $Y$ is $H$-cocompact. By Lemma~\ref{lem:pushout}, the pushout $Z'$ is $H$-cocompact. Lemma~\ref{lem:folding} implies that $Z'\to Z$ is surjective. It follows that $Z$ is $H$-cocompact. 
Suppose that $Y$ is connected. Then $Z'$ is connected by Lemma~\ref{lem:pushout}. Since $Z'\to Z$ is surjective, $Z$ is connected.
\end{proof}

For the lemmas below, consider the natural map $I^0(H,Y) \to I^0(H, Z)$ induced by $Y\to Z$. 
\begin{lem}
The image of $\isolated (H, Y)$ contains $\isolated (H, Z)$.
\end{lem}
\begin{proof}
Suppose that $v \in \isolated (H, Z)$ and let $z \in Z$ be a representative.  Let $e$ be the $1$-cell of $Z$ giving rise to the isolated vertex of $\link (z, Z)$.  Since $Z'\to Z$ is surjective, there is $1$-cell $e'$ of $Z'$ mapping to $e$, and a corresponding $0$-cell $z'$ mapping to $z$. Since $\link (z, Z)$ has an isolated vertex induced by $e$, $\link (z', Z')$ has an isolated vertex $s$ induced by $e'$. 
Suppose that $e'$ has a preimage $f$ in $\bigcup_HD$. Then $f$ is on the boundary path of a component of $\bigcup_HD$. It follows that $e'$ has a preimage in $\bigcup_HP$, and hence $e'$ has a preimage in $Y$. Therefore there is $y \in Y$ such that the image of $\link (y, Y)$ in $\link (z', Z')$ contains the isolated vertex $s$. It follows that $\link (y, Y)$ has an isolated vertex and therefore there is $u \in \isolated (H, Y)$ mapping to $v$. 
\end{proof}

\begin{lem}
The   image of $\bnd (H, Y) \cup \isolated (H, Y)$ contains $\bnd (H, Z)$. 
\end{lem}
\begin{proof}
Suppose that $v \in \bnd (H, Z)$. Let $z \in Z$ be a representative of $v$.  If $\link (z, Z)$ is a single point, then $v \in \isolated (H,Y)$ and the previous lemma shows that  there is $u \in \isolated (H, Y)$ that maps to $v$.  Consider the case that $\link (z, Z)$ has a spur $s$ with terminal vertex $t$. 

Suppose that $s$ corresponds to a $2$-cell in the image of $Y\to Z$. Then there is $y\in Y$ mapping to $z$ such that the image of $\link (y, Y) \to \link (z, Z)$ contains $s$. Since $Y\to Z$ is an immersion, $\link (y, Y)$ is a subgraph of $\link (z, Z)$ and hence it has a spur.  In particular, there is $u\in \bnd (H, Y)$ that maps to $v$.

Otherwise, there is a $0$-cell $w$ of $\bigcup_H D$ that maps to $z$, and the image of  $\link (w, \bigcup_H D) \to \link (z, Z)$  contains $s$.  In particular, $\link (w, \bigcup_H D)$ has a spur and therefore $w$ is in the boundary of a connected component of $\bigcup_H D$. Therefore, there is also $y \in Y$ that maps to $z$, and the image of $\link(y, Y)\to \link(z, Z)$ contains $t$. As $\bigcup_HD \to Z$ and $Y\to Z$ are immersions, $\link (y, Y)$ and $\link (w, \bigcup_H D)$ are subgraphs of $\link (z, Z)$.  Both subgraphs contain the vertex $t$, but $\link (y, Y)$ does not contain $s$. It follows that $t$ is an isolated vertex of  $\link (y, Y)$. Therefore, there is $u \in \isolated (H,Y)$ that maps to $v$.
\end{proof}

The two lemmas above imply that  the image of $\bnd (H, Y) \cup \isolated (H, Y)$  contains $\bnd (H, Z) \cup \isolated (H, Z)$, and this concludes the proof of Theorem~\ref{thm:collapsing-essential-paths}. 

\subsection{No self-immersions}

\begin{lem}[No Self-immersions]\label{lem:no-self-immersion}
Let $X$ be a $G$-cocompact, proper and connected complex. Any $G$-equivariant immersion $\phi\colon X\to X$ is an isomorphism. 
\end{lem}
\begin{proof}
The quotient space $X/G$ is a combinatorial complex, and $\phi$ induces a self-immersion $\psi \colon X/G \to X/G$. Since $X/G$ is compact and connected,  $\psi$ is an isomorphism~\cite[Lem. 6.3]{Wi04} and hence $\phi$ is onto. 

Let $u$ be a $0$-cell. Since $\psi$ is an isomorphism, all elements of $\phi^{-1}(u)$ are $G$-equivalent. Therefore $|\phi^{-1}(u)|$  is a lower bound for the size of the $G$-stabilizer of $u$.  Since $X$ is $G$-cocompact and proper, there is an upper bound on the cardinality of cell stabilizers. 
Therefore $|\phi^{-1}(\phi^n(u))|=1$  for some $n>0$ depending on $u$.  By cocompactness, there is $m>0$ such that $\phi$ restricted to $\phi^{m}(X)=X$ is injective.  
\end{proof}

\section{Existence of cores}\label{sec:core}

This section contains the proof of the Theorem~\ref{thm:core}. The section is divided into five subsections. The first four subsections contain preliminary results, and the last subsection discusses the proof of the theorem. 

\subsection{Angled Graphs}

\begin{lem}[Curvature and Connected Components]\label{lem:components-weight} \label{cor:components-weight}
Let $\Delta$ be a cocompact angled $H$-graph.  Let $\Delta_1, \dots , \Delta_\ell$ be a collection of representatives of $H$-orbits of connected components of $\Delta$, and let $H_i$ be the stabilizer of $\Delta_i$. If  $\standardcurvature (H_i, \Delta_i) \leq \pi \cdot |H_i|^{-1}$ for $1\leq i\leq \ell$, then $\standardcurvature (H, \Delta) \leq \standardcurvature (H_i, \Delta_i)$ for $1\leq i\leq \ell$.
\end{lem}
\begin{proof}
First notice that 
\begin{equation} \label{eq:weight-components}
\standardcurvature (H, \Delta) =   \sum_{i=1}^\ell \standardcurvature (H_i, \Delta_i) + 2\pi \cdot \left ( |H|^{-1} - \sum_{i=1}^\ell |H_i|^{-1}  \right ).
\end{equation}

If $\ell=1$ and $H=H_1$, then $\Delta=\Delta_1$ and obviously $\curvature (H, \Delta) = \curvature (H_1, \Delta_1)$. 
Otherwise $2 |H|^{-1} \leq \sum_{i=1}^\ell |H_i|^{-1}$  and hence equation~\eqref{eq:weight-components} implies that
\[ \standardcurvature (H, \Delta) \leq \sum_{i=1}^\ell \left ( \standardcurvature (H_i, \Delta_i) - \pi \cdot |H_i|^{-1}  \right ).\]
Since $\standardcurvature (H_i, \Delta_i) \leq \pi \cdot |H_i|^{-1}$, it follows that $\standardcurvature (H, \Delta) \leq \standardcurvature (H_i, \Delta_i)$.
\end{proof}

\begin{prop}\label{prop:weight-spurless-sections}
Let $\Gamma$ be an angled $G$-graph such that each for each subgroup $K\leq G$, every regular $K$-section has curvature $\leq \alpha \leq 0$.  Suppose $H$ is a subgroup of $G$ and $\Delta$ is a non-empty and spurless $H$-section of $\Gamma$.  The following statements hold.
\begin{enumerate}
\item \label{prop:stat-4} If $\Delta$ contains an edge, then $\standardcurvature (H, \Delta) \leq \alpha$.
\item \label{prop:stat-1} $\standardcurvature (H, \Delta)>0$ if and only if $\Delta$ is a single vertex and $H$ is a finite group.
\end{enumerate}
\end{prop}
\begin{proof}
Let $\Delta_i$, $H_i$ and $\ell$ be as in the statement of Lemma~\ref{lem:components-weight}. Since $\Delta$ is spurless, for each $i$ either $\Delta_i$  is a regular $H_i$-section of $\Delta$ or else $\Delta_i$ is a single vertex. It follows that $\standardcurvature (H_i, \Delta_i) \leq 0 \leq \pi |H_i|^{-1}$ for each $i$.

Suppose $\Delta$ contains an edge. Without loss of generality assume that $\Delta_1$ is a connected component with at least one edge. Then $\Delta_1$ is a regular $H_1$-section, and by Lemma~\ref{cor:components-weight},  $\standardcurvature (H, \Delta) \leq \standardcurvature (H_1, \Delta_1)\leq \alpha$.

Suppose that $\standardcurvature (H, \Delta)>0$. Since $\alpha \leq 0$, Lemma~\ref{cor:components-weight} implies that each connected component of $\Delta$ is a single point.  Therefore $0<\standardcurvature (H, \Delta) = 2 \pi |H|^{-1} - \pi \sum_{i=1}^\ell |H_i|^{-1}$. This implies that $\ell=1$, $H$ is finite group and $H_1=H$. In particular, $\Delta$ is a single point and $H$ is a finite group.  The  ``if"  part of statement~\eqref{prop:stat-1} is immediate.
\end{proof}

\begin{cor}\label{cor:zero-weight}
Let $\Gamma$ be an angled $G$-graph such that each regular section has curvature $\leq \alpha< 0$. 
Suppose that $\Delta$ is an $H$-section of $\Gamma$ such that $\Delta$ is a non-empty and spurless.
Then $\standardcurvature (H, \Delta) = 0$ if and only if $\Delta$ is an edgeless graph and either:
\begin{enumerate}
\item  $\Delta$ consists of two vertices and $H$ acts trivially on $\Delta$, or
\item  the stabilizer of each vertex of $\Delta$ is infinite.
\end{enumerate}
\end{cor}
\begin{proof}
Let $\Delta_i$, $H_i$ and $\ell$ be as in the statement of Lemma~\ref{lem:components-weight}.
Suppose that $\standardcurvature (H, \Delta) = 0$. Since $\alpha<0$,  Proposition~\ref{prop:weight-spurless-sections} implies that each $\Delta_i$ is a single vertex.
Then 
\begin{equation*}0 = \standardcurvature (H, \Delta) = 2\pi\cdot  |H|^{-1} - \pi \cdot \sum_{i=1}^\ell |H_i|^{-1}.
 \end{equation*}
If $H$ is an infinite group then each $H_i$ is infinite, and therefore the stabilizer of each point of $\Delta$ is an infinite subgroup of $H$.
If $H$ is a finite group, then $2  = \sum_{i=1}^\ell [H: H_i]$, where $[H: H_i]$ is the index of $H_i$ in $H$.  Hence $\ell=2$ and $H_i=H$ for $i=1,2$,
and, in particular, the action of $H$ on $\Delta$ is trivial.

The ``if" part of the statement follows by a direct computation of $\standardcurvature (H, \Delta)$ and  is left to the reader.
\end{proof}

\subsection{Counting Immersions in Nonpositively curved complexes}

\begin{defn}[$G$-equivalent maps]
Let $X$ be a $G$-complex, let $H$ be a subgroup of $G$, and, for $i=1,2$, let $Y_i$ be an $H$-complex.  A pair of $H$-equivariant immersions $\phi_1 \colon Y_1 \to X$ and $\phi_2\colon Y_2\to X$ are {\em $G$-equivalent} if there is an $H$-isomorphism of complexes $\psi \colon Y_1 \to Y_2$ and $g \in G$ such that $g \circ \phi_1  = \phi_2 \circ \psi$.
\end{defn}

\begin{defn}
Let $X$ be a nontrivial cocompact and proper angled $G$-complex. Define \begin{itemize}
\item $\zero (G, X)$ as the set of $v\in I^0(G, X)$ with $\curvature (v)=0$, 
\item $\nega (G, X)$ as the set of $v\in I^0(G, X)$ with $\curvature (v)<0$, and
\item $\positivecells (G, X)$ as the set of $v\in I^0(G, X)$ with $\curvature (v)>0$.
\end{itemize}
\end{defn}

\begin{thm}[Counting Immersions in Nonpositively Curved $G$-Complexes] \label{thm:main-counting}
Let $X$ be a nontrivial cocompact and proper angled $G$-complex with sectional curvature $\leq 0$. 
Let $H$ be a subgroup of $G$, let $r$, $s$, and $t$ be fixed numbers.
Up to $G$-equivalence,  there are finitely many  $H$-equivariant immersions $Y\to X$ with the following properties:
\begin{enumerate}
\item $Y$ is $H$-cocompact and connected.
\item $\chi (H, Y) \geq r$.
\item $|\zero(H,Y) |\leq s$.
\item $|\bnd(H,Y) |\leq t$.
\end{enumerate}
\end{thm}

%

\begin{lem}[Immersions Determined by a Compact Complex]\label{lem:fundamental-domain}
Let $X$ be a proper $G$-complex, let $H$ be a subgroup of $G$, let $K$ be a finite connected complex, and let $\psi \colon K\to X$ be an immersion. There are finitely many $G$-equivalence classes of $H$-equivariant immersions $\phi\colon Y\to X$  such that there is an embedding $\imath\colon K\hookrightarrow Y$ satisfying $\phi \circ \imath = \psi$ and the $H$-translates of $K$ cover $Y$.
\end{lem}
\begin{proof}
Observe that if $\phi\colon Y\to X$ is an $H$-equivariant immersion and $K$ is a subcomplex with $\bigcup_{h\in H} K=Y$ then $\phi$ is completely determined by its restriction to $K$. 

The $H$-proper complex $Y$ is completely determined  by the finite set of elements  $\{g\in H : K\cap gK \neq \emptyset \}$ of $H$ and the isomorphisms between the complexes $J_g=g^{-1}K \cap K$    and $J_g'=K\cap g K$. Indeed, one can recover $Y$ by taking $H\times K$ and identifying the various
$h\times J_g$ with $hg \times J_g'$ using the isomorphism.

To show that there are finitely many  possibilities for the above data we argue as follows: let $g_1, \ldots, g_n$ be the set of elements of $H$ such that $\psi K \cap g \psi K \not =\emptyset$, and note that this set is finite since $G$ acts properly on $X$. For each $i$, there are finitely many choices of isomorphisms between subcomplexes $J_i \subset K$ and $J_i'\subset K$.
\end{proof}

%

\begin{lem}[Counting Immersions] \label{lem:counting-immersions}
Let $X$ be a proper, cocompact and connected $G$-complex, let $H$ be a subgroup of $G$, and let $M$ be a positive integer. Up to $G$-equivalence, there are finitely many $H$-equivariant immersions $Y \to X$ such that $Y$ is connected and $|I^0(H,Y)|<M$. 
\end{lem}
\begin{proof}
Observe that every $H$-cocompact complex $Y$ with $|I^0(H,Y)|<M$ contains a connected subcomplex $K$ with $|K^0|<M$. Since $X$ is proper and $G$-cocompact, there are finitely many $G$-equivalent classes of immersions $K\to X$ where $K$ is connected and $|K^0|<M$. The result then follows from Lemma~\ref{lem:fundamental-domain}.
\end{proof}

\begin{lem} \label{lem:zero-cells}
Let $X$ be a nontrivial cocompact and proper angled $G$-complex with generalized sectional curvature $\leq 0$. Then \[I^0(G, X) = \zero (G, X) \cup \nega (G, X) \cup \bnd (G, X).\] 
\end{lem}
\begin{proof}
Let $v \in I^0(G, X)$, suppose that $v \not \in \bnd (G, X)$.
Since $X$ is nontrivial and connected, $\link (x, X)$ is non-empty.  
Since $v \not \in \bnd (G, X)$, $\link (x)$ is spurless and not a single point.
Since $X$ has generalized sectional curvature $\leq 0$, Proposition~\ref{prop:weight-spurless-sections} implies that  $\standardcurvature (G_x, \link (x, X))\leq 0$. Hence $v \in \zero (G, X) \cup \nega (G, X)$.
\end{proof}

\begin{defn}[$\negativebound (G, X)$, $\positivebound (G, X)$]\label{def:Bounds}
Let $X$ be a cocompact and proper angled $G$-complex. There are finitely many 0-cells of $X$ up to the action of $G$. Since $X$ is proper and cocompact, for each $0$-cell $x$ of $X$, the angled $G_x$-graph $\link (x)$ has finitely many sections.  Let $\negativebound (G, X)$ be the maximum curvature among all possible such sections with negative curvature; if there are no sections with negative curvature let $\negativebound (G, X)=-1$. Analogously, let $\positivebound (G, X)$ be the maximum curvature among all sections with non-negative curvature; if there are no sections with non-negative curvature let $\positivebound (G, X)=0$.
\end{defn}

\begin{rem}
For a cocompact and proper angled $G$-complex $X$, $\negativebound (G, X)<0 \leq \positivebound (G, X)$.
\end{rem}

\begin{lem}\label{lem:Bounds} 
Let $X$ be a cocompact and proper angled $G$-complex. Suppose that $\curvature (f)\leq 0$ for each $2$-cell of $X$. 
If $H$ is a subgroup of $G$, $Y$ is a cocompact $H$-complex, and $Y\to X$ is an $H$-equivariant immersion, then  
\begin{equation*} |\nega (H, Y)| \leq   \frac{2\pi \cdot \chi (H, Y) - \positivebound (G, X) \cdot |\positivecells (H, Y)| }{ \negativebound (G, X)}, \end{equation*} 
\end{lem}
\begin{proof}
For $v\in I^0(H, Y)$, $\curvature (v) = \standardcurvature (H_y, \link (y, Y))$ where $y$ is a representative of $v$. Let $x \in X$ be the image of $y$. Since $Y\to X$ is an immersion, $\link (y, Y)$ is an $H_y$-section of  the $G_x$-graph $\link (y, X)$. In particular, $\curvature (v) \leq \negativebound(G, X)<0$ or $0 \leq \curvature (v)\leq \positivebound(G, X)$.

By the Gauss-Bonnet Theorem~\ref{thm:Gauss-Bonnet} and the assumption that $2$-cells have non-positive curvature, 
\[ 2\pi \cdot \chi (H, Y) \leq \negativebound(G, X) \cdot | \nega (H, Y)| +  \positivebound (G, X) \cdot |\positivecells (H, Y)|.\]
Since $\negativebound (G, X)<0$, the conclusion of the lemma is immediate. 
\end{proof}

\begin{proof}[Proof of Theorem~\ref{thm:main-counting}]
Let $Y \to X$ be an $H$-equivariant immersion satisfying the listed properties. 
By Proposition~\ref{prop:standard-generalized}, $X$ has generalized sectional curvature $\leq 0$. Therefore $Y$  has generalized sectional curvature $\leq 0$ as well.  By Lemma~\ref{lem:zero-cells},  $ I^0 (H, Y) = \zero (H, Y) \cup \nega (H, Y)  \cup \bnd (H, Y)$.
By Lemmas~\ref{lem:zero-cells} and~\ref{lem:Bounds}, 
\begin{equation*} 
\begin{split}
|I^0 (H, Y)| & \leq  |\zero (H, Y)|  +  |\bnd (H, Y)|  + |\nega (H, Y)|  \\
& \leq  s + t + \frac{ 2\pi \cdot r- \positivebound(G, X) \cdot t } { \negativebound (G, X)}.
\end{split} \end{equation*}
Since there is an upper bound for the size of $I^0(H, Y)$ independent of $H$ and $Y$,  Lemma~\ref{lem:counting-immersions} implies that there are finitely many posibilities for $Y\to X$.
\end{proof}

\subsection{Proof of the Simply-connected Core Theorem~\ref{thm:core}}

Construct a sequence of $H$-equivariant immersions $\phi_n\colon Y_{n-1} \to Y_{n}$ for $n\geq 1$ and $\psi_n\colon Y_n \to X$ for $n\geq 0$ in the following way. 

\begin{equation*}
\xymatrix{
Y_0 \ar[r]^{\phi_1} \ar[dd]_{\psi_0}& Y_1\ar[r]^{\phi_2} \ar[ldd]_{\psi_1} & \cdots \ar[r]^{\phi_{n}} & Y_n \ar[llldd]_{\psi_n} \ar[r]^{\phi_{n+1}} & \cdots &  \\
\\
X & &  & & & & 
}
\end{equation*}

Let $Y_0=Y$ and let $\psi_0\colon Y_0\to X$ be the inclusion map.  Assume that  $\psi_n\colon Y_n \to X$ has been defined.  Suppose there is a nontrivial path $P\to Y_n$ such that either
\begin{itemize}
\item $P\to Y_n$ is a closed path that is not null-homotopic, or 
\item $P \to Y_n$ is not a closed path but $P \to Y_n\to X$ is a closed path. 
\end{itemize}
Such path is called {\em essential}. Choose a  path $P\to Y_n$ of minimal length with the above property. Observe that $P\to Y_n \to X$ is a simple closed path by minimality.  
Let $\phi_{n+1} \colon Y_n \to Y_{n+1}$ be the $\pi_1$-surjective map and  $\psi_{n+1} \colon Y_{n+1}\to X$ be the $H$-equivariant immersion provided by Theorem~\ref{thm:collapsing-essential-paths} applied to the immersion $\psi_n\colon Y_n\to X$ and the path $P\to Y_n$.  
If there is no path $P\to Y_n$ as above, then the sequence stabilizes in the sense that  $Y_{m}=Y_n$,  $\psi_{m}=\psi_n$ and $\phi_{m}$ be the identity map for all $m\geq n+1$. (We show below that the sequence always stabilizes.) 

\begin{lem}
$Y_n$ is connected and $H$-cocompact.
\end{lem}
\begin{proof}
This follows by induction using Theorem~\ref{thm:collapsing-essential-paths} with base case that $Y_0$ is connected and $H$-cocompact.
\end{proof}

\begin{lem}\label{lem:pi_1-surjective}
 $\phi_n$ is an immersion and is $\pi_1$-surjective.
\end{lem}
\begin{proof}
The property of $\pi_1$-surjectivity follows from the definition of $\phi_n$ in terms of Theorem~\ref{thm:collapsing-essential-paths}. Since $\psi_{n+1} \circ \phi_{n+1} = \psi_n$ and $\psi_n$ is an immersion, $\phi_{n+1}$ is an immersion.  
\end{proof}

\begin{lem}
$|\zero (H, Y_n) \cup \bnd (H, Y_n)| \leq |\isolated (H, Y_0) \cup \bnd (H, Y_0)|$.
\end{lem}
\begin{proof}
By Theorem~\ref{thm:collapsing-essential-paths}~\eqref{thm:collapsing-3}, an induction argument shows that $|\isolated (H, Y_n) \cup \bnd (H, Y_n)| \leq |\isolated (H, Y_0) \cup \bnd (H, Y_0)|$. Therefore it is enough to prove that $\zero (H, Y_n) \cup \bnd (H, Y_n)$ is a subset of $\isolated (H, Y_n) \cup \bnd (H, Y_n)$. Let $v \in \zero (H, Y_n) - \bnd (H, Y_n) $ and let $y$ be a representative of $v$.
Since $Y_n\to X$ is an immersion, $Y_n$ has negative sectional curvature. By Corollary~\ref{cor:zero-weight}, $\link (y, Y_n)$ consists of two vertices and no edge. Thus $v \in \isolated (H, Y_n)$.
\end{proof}

\begin{lem}
$\chi (H, Y_n) \geq  -b_1^{(2)}(H, Y_0)$
\end{lem}
\begin{proof}
By Lemma~\ref{lem:pi_1-surjective}
 and an induction argument, $Y_0 \to Y_{n+1}$ is $\pi_1$-surjective. Then Corollary~\ref{cor:betti-goal} implies that $b_1^{(2)}(H, Y_0) \geq b_1^{(2)}(H, Y_n)$. Then the conclusion follows from $\chi (H, Y_n) \geq - b_1^{(2)}(H, Y_n)$.
\end{proof} 

For positive integers $m<n$, let $\phi_{m,n}$ denote the immersion $\phi_{n-1}\circ \cdots \circ \phi_{m+1}$ from $Y_m$ to $Y_n$.

\begin{lem}\label{lem:conclusion}
There is $m_0>0$ such that $\phi_{m_0, n}$ is an $H$-equivariant isomorphism for every $n> m_0$. 
\end{lem}
\begin{proof}
By Theorem~\ref{thm:main-counting} and the previous lemmas, the sequence $\{ \psi_n \}_{n\geq 1}$ contains only finitely many non $G$-equivalent immersions. If the statement is false, then there are positive integers $m<n$ such that $Y_m\to X$ and $Y_{n}\to X$ are $G$-equivalent but the $H$-equivariant immersion $\phi_{m, n}\colon Y_m \to Y_n$ is not an isomorphism. Since $Y_m$ and $Y_n$ are isomorphic as $H$-complexes, $\phi_{m, n}$ would be a self-immersion that is not an isomorphism contradicting  Lemma~\ref{lem:no-self-immersion}. 
\end{proof}

{\bf Conclusion of the proof of Theorem~\ref{thm:core}.} By Lemma~\ref{lem:conclusion}, it follows that there exists $m_0>0$ such that $Y_{m_0}$ has no essential paths as defined in the construction of the $Y_i$'s .  In particular, $\psi_{m_0}$ is an embedding and $Y_{m_0}$ is simply-connected. Therefore $\psi_{m_0}(Y_{m_0})$ is a simply-connected $H$-cocompact subcomplex of $X$ containing $Y_0$.

\section{Quasiconvex Cores}\label{sec:qccore}

\begin{figure}\centering
\includegraphics[width=.6\textwidth]{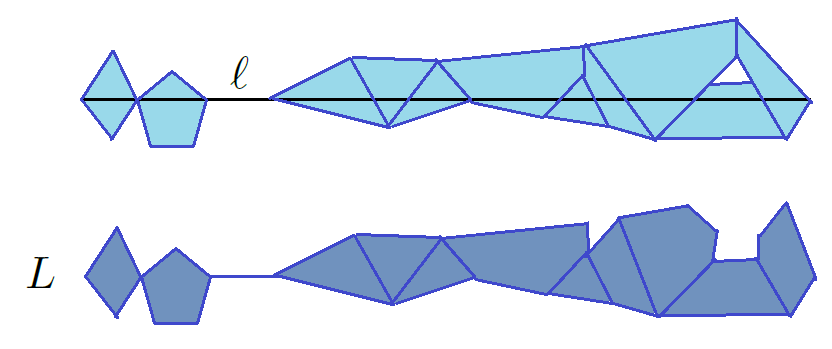}
\caption{At the top is the sequence of open cells of $X$ intersecting the geodesic $\ell$. At the bottom is the resulting complex $L$. } \label{fig:carrier}
\end{figure}

This section contains the proof of Theorem~\ref{thm:qccore}. Let $X$ be a proper cocompact  $CAT(0)$ $G$-complex whose cells are convex. This is a complete geodesic metric space~\cite{Br92}; for background on $CAT(0)$ cell-complexes we refer the reader to~\cite{BrHa99}. 

The proof is split in several subsections.  Assign angles as they arise from the $CAT(0)$-metric and suppose $X$ has negative sectional curvature.  Let $H$ be a subgroup of $G$ and let $Y$ be a simply-connected cocompact $H$-subcomplex. Let $\ell$ be a geodesic segment in $X$ such that its endpoints are $0$-cells of $Y$.  

\subsection{The carrier $L$ of $\ell$, the paths $P_L$, $P_Y$ and the disk diagram $D$.}

Let $R_1, \ldots, R_n$ be the sequence of open cells of $X$ that intersect $\ell$ in the order in which they are traversed by $\ell$. Since cells of $X$ are convex  there are no repetitions in the sequence and $\partial R_i \cap \partial R_j$ is connected. Let $L$ be the complex constructed by taking the disjoint union of closures $\bar R_1 \sqcup \cdots \sqcup \bar R_n$ and identifying the two copies of $\partial  R_i \cap \partial R_{i+1}$ in $\bar R_i$ and $\bar R_{i+1}$ for $1\leq i<n$. Observe that $L \to X$ is a near-immersion, $L$ is simply connected, and  $\ell \hookrightarrow X$ factors as $\ell \to L \to X$. See Figure~\ref{fig:carrier}.

Since $Y$ and $L$ are connected, there are edge paths $P_Y\to Y$ and $P_L \to L$ connecting the endpoints $s$ and $t$ of $\ell$. Since $X$ is simply-connected there is a disk diagram $D \to X$ between $P_Y\to X$ and $P_L \to X$.  Choose $P_Y\to Y$ and $P_L \to L$ and $D \to X$ so that  $\big ( \area(D), |\partial D| \big)$  is minimal in the lexicographical order. 

\begin{lem}[$\ell$ is uniformly close to $P_L$]\label{lem:PLell}
If $R\to X$ is an open cell intersecting $\ell$, then its closure  $\bar R\to X$ intersects the image of $P_L\to X$. 
\end{lem}
\begin{proof}
Since $P_L\to L$ connects the endpoints of $\ell \to L$, if a closed cell $S$ disconnects $L$ then $P_L\to L$ intersects $S$. By definition of $L$, if $R$ is an open cell of $X$ intersecting $\ell$ then the closure of $R$ in $L$ disconnects $L$. 
\end{proof}

\begin{defn}[Cut $0$-Cells and Cut-Components] 
\label{def:cut-tree}
A $0$-cell $v$ is called a
cut $0$-cell of $D$ provided that $D -\{ v\}$ is not connected.
 Let $V$ be the set of all cut $0$-cells of $D$. Closures of connected components of $D -V$ are \emph{cut-components}. A cut-component is \emph{non-singular} if it contains a 2-cell.     
\end{defn}

\begin{lem}\label{lem:PYembedded}
The path $P_Y\to D$ is embedded.
\end{lem}
\begin{proof}
Suppose that $P_Y\to D$ is not embedded. Then $P_Y$ is a concatenation $U_1VU_2$ such that (a) the terminal point of $U_1\to P_L \to  D$ and the initial point of $U_2\to P_L \to  D$ is a $0$-cell $u$  and  (b) $V\to P_Y \to  D$ is a nontrivial path with no internal $0$-cells in common with $P_L\to D$. Observe that $u$ is a cut $0$-cell of $D$. Let $\dot P_Y \to D$  be the path $U_1U_2 \to D$ and observe that $\dot P_Y \to D \to X$ factors through $Y\to X$ and $\dot P_Y \to X$ connects the endpoints of $\ell$. The paths $\dot P_Y \to D$ and $P_L \to D$ bound  a subdiagram $\dot D$ of $D$. Since $V\to D$ is nontrivial, the complexity of $\dot D$ is strictly smaller than the complexity of $D$. Then $\dot P_Y$ and $\dot D$ violate the minimality of $D$.
\end{proof}

\begin{figure}
\includegraphics[width=.35\textwidth]{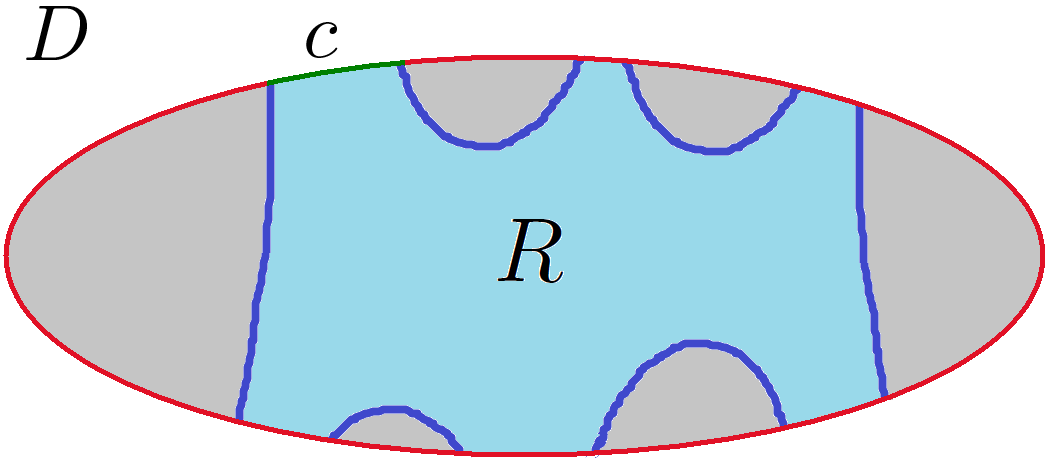} \hspace{.1\textwidth}
\includegraphics[width=.35\textwidth]{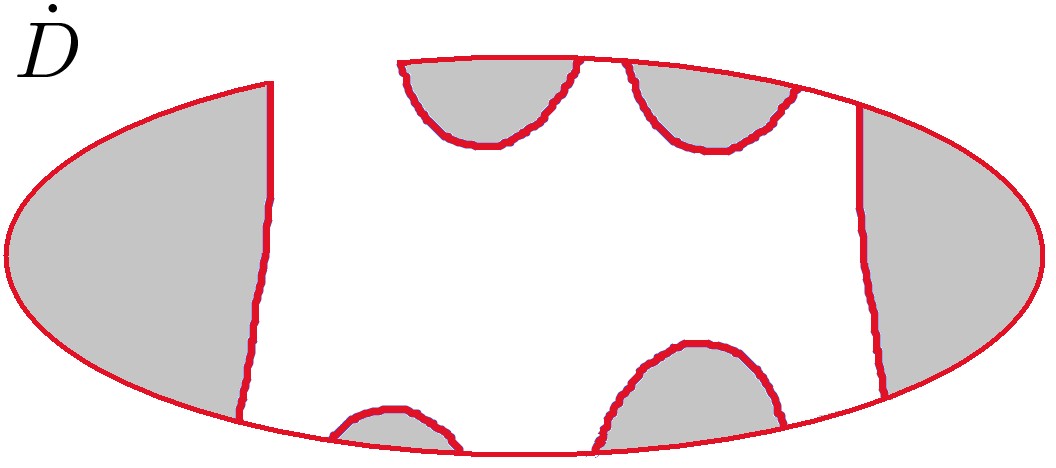}  \label{fig:2cell-relevant-cut.eps}
\caption{If a cut-component of $D$ has a $2$-cell $R\to D$ with two different boundary arcs intersecting $\ell$, then $D$ has no minimal complexity. Illustration of the reduction of complexity in the proof of Lemma~\ref{lem:reduction-1} \label{fig:reduction-2cell}}
\end{figure}

\begin{lem}[$2$-cells of $D$ intersecting $P_L$]\label{lem:reduction-1}
If $R\to D$ is a $2$-cell such that $\partial R\to D$ intersects $P_L\to D$, then the interior of $R\to X$ does not intersect $\ell$.
\end{lem}
\begin{proof}
If the interior of $R\to X$ intersects $\ell$, then $\partial R\to X$ factors through $L\to X$.  Since $\partial R\to D$ intersects $P_L\to D$, it follows that $P_L$ is a concatenation $U_1 c U_2$ where $c$ is a cell mapped into  $\partial R\to D$. In particular, $\partial R$ is a concatenation $cQ$ where $Q\to \partial R$ is a path; if $c$ is a $0$-cell then $c$ is the initial point of the path $Q$. Let $\dot D$ be the subdiagram of $D$ obtained by removing the interiors of $c$ and $R$; in the case that $c$ is a $0$-cell, remove the interior of $R$, remove $c$, and add two copies of $c$ to obtain a simply connected diagram.  Let  $\dot P_L \to L$ be the path $U_1 Q^{-1} U_2$ (or $U_1 Q U_2$) and observe that $\dot P_L$ and $\dot D$ violate the minimality of $D$.  See Figure~\ref{fig:reduction-2cell}.
\end{proof}

\begin{lem}[Internal Paths in $D$]\label{lem:no-internal-paths}
If there is a nontrivial internal path $T\to D$, as defined in~\ref{defn:internal}, such that $T\to D\to X$ factors through $L\to X$ and has  endpoints in $P_L\to D$ then $D$ does not have minimal complexity.
\end{lem}
\begin{proof}
Suppose that $T\to D$ is an internal path satisfying the hypothesis of the lemma. Without loss of generality, we can assume that $T\to D$ is an embedding; indeed its image contains an internal and embedded path $T'\to D$ with the same endpoints as $T\to D$.
The endpoints of $T$ split the boundary path of $D$ as a concatenation of paths $UV\to \partial D$ such that $T$ and $U$ have the same endpoints and  $P_Y\to D$ is a subpath of $U^{-1}$ (the path $V$ may be trivial). Then $P_L\to D$  is a concatenation $U_1VU_2$ where each $U_i$ is a subpath of $U$.
Let $\dot P_L \to D$  be the path $U_1T^{-1}U_2$ and observe that $\dot P_L \to D \to X$ factors through $L\to X$. The paths $P_Y \to D$ and $\dot P_L \to D$ bound  a subdiagram $\dot D$ of $D$. Since $T\to D$ is internal, the complexity of $\dot D$ is strictly smaller than the complexity of $D$. Then $\dot P_L$ and $\dot D$ violate the minimality of $D$. See Figure~\ref{fig:reduction-internal-path}.
\end{proof}

\begin{lem}[$2$-cells of $D$ intersecting $\ell$]\label{lem:at-most-one-l-component}
Let $R\to D$ be a $2$-cell.  Suppose $\partial R\to D$ is a concatenation $S_1T_1 \cdots S_mT_m$ where each $S_i\to D$ is a  boundary path $S_i \to \partial D$, and each $T_i\to D$ is a nontrivial internal path in $D$.
Then at most one subpath $S_i\to X$ intersects $\ell$ and factors through $P_L\to X$.
\end{lem}
\begin{proof}
Suppose that two different paths $S_i\to X$ intersect  $\ell$ and factor  through $P_L\to X$. By convexity of $R$, either the interior of $R\to X$ intersects $\ell$ or else there is a path $T_i\to D$ which maps into $\ell$. The former case is impossible by
Lemma~\ref{lem:reduction-1} and the latter case is impossible by Lemma~\ref{lem:no-internal-paths}.
\end{proof}

\begin{figure}
\includegraphics[width=.65\textwidth]{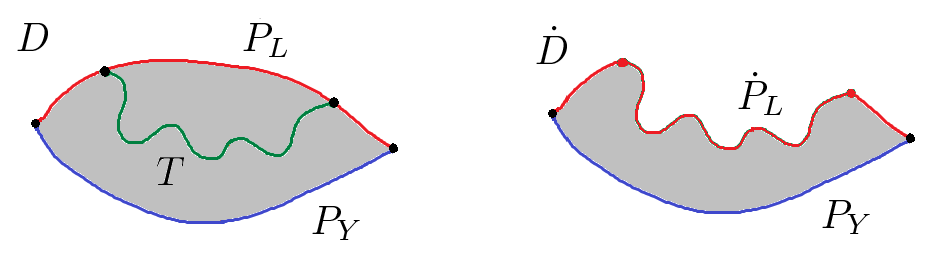}
\caption{Illustration of the reduction of complexity in the proof of Lemma~\ref{lem:no-internal-paths}.\label{fig:reduction-internal-path}}
\end{figure}

\begin{lem}[Cut $0$-cells of $D$]\label{lem:cut-cells-2}
Let  $u$ be a cut $0$-cell of $D$.
\begin{enumerate}
\item \label{lem:cc1} If  $u$ is in $P_Y\to D$ then $u$ is also in $P_L \to D$.
\item \label{lem:cc2}  If $u\to D \to X$ is contained in $\ell$, then $u$ is in $P_Y\to D$.
\end{enumerate}
\end{lem}
\begin{proof}
 $P_Y\to D$ is embedded by Lemma~\ref{lem:PYembedded} and hence the first statement is immediate.  To prove the second statement, suppose the image of $u$ in $X$ is contained in $\ell$ and $u$ is not in $P_Y\to D$. Hence $u$ is in $P_L\to D$. Observe that all points of $P_L$ that map to  $u \in D$ map to the same point in $L$. This follows since $\ell \to X$ is injective, contains the cell $u\to X$,  and factors through $L\to X$. Let $V\to D$ be the subpath of $P_L\to D$ starting and ending at $u$ and traveling around the components of $D-U$ not containing $P_Y\to D$.  Express $P_L$ as a concatenation $U_1VU_2$ and note that $V\to   D$ and $P_Y\to D$ are disjoint.  Let $\dot P_L$ be the path $U_1U_2$. Observe that $\dot P_L \to X$ is a path connecting the endpoints of $\ell$, factors through $L\to X$, and factors through $P_L\to D$. Moreover, the paths $\dot P_L \to D$ and $P_Y \to D$ bound  a subdiagram $\dot D$ of $D$. Since $V\to D$ is nontrivial, the complexity of $\dot D$ is strictly smaller than the one of $D$. Then $\dot P_L$ and $\dot D$ violate the minimality of $D$.
\end{proof}

\subsection{Good Path in $D$}

The main goal of this subsection is to prove Proposition~\ref{prop:qcmain1}.

\begin{defn}[Good paths in $D$]\label{defn:goodpaths}
A path $Q\to D$  is a \em{good path} if for every $0$-cell $c$ in the interior of the path $Q\to D$, either $c$ is an internal cell of the diagram $D$, or the image of $c\to D\to X$ does not intersect $\ell$.
\end{defn}

\begin{prop}[Good paths with terminal point in $P_Y$]\label{prop:qcmain1}
Let $u$ be a $0$-cell of $\partial D$. Suppose $u$ is not an internal $0$-cell of a boundary arc of $D$.  Then there is a good path $Q\to D$ from $u$ to a $0$-cell of $P_Y\to D$.
\end{prop}

The following lemma is immediate.
\begin{lem}[Concatenation of good paths]\label{lem:concatenation-goodpaths}
Let $P_1\to D$ and $P_2\to D$ be good paths such that the terminal point $w$ of $P_1$ equals the initial point of $P_2$. Suppose the image of $w$ in $D\to X$ does not intersect $\ell$. Then the concatenation $P_1P_2\to D$ is a good path.
\end{lem}

\begin{lem}[Good Paths along a Boundary of a 2-Cell]\label{lem:basic-good-path}
Let $R\to D$ be a $2$-cell such that $\partial R$ is a concatenation $S_1T_1 \cdots S_mT_m$ where each $S_i\to D$ is a  boundary arc $S_i \to \partial D$ and each $T_i\to D$ is a nontrivial internal path in $D$.
\begin{enumerate}
\item \label{lem:basic(1)} Suppose each $S_i\to D$ factors through $P_L\to D$.   
If $u$ and $v$ are distinct $0$-cells in the intersection of $\partial R\to D$ and $\partial D$ and if $u$ and $v$ are not internal $0$-cells of a boundary arc of $D$, then there is a good path  between $u$ and $v$. 
\item \label{lem:basic(2)}Suppose $\partial R \to D$ intersects $P_L\to D$ but some  $S_i\to D$ does not factor through $P_L\to D$.
If $u$ is a $0$-cell in the intersection of $\partial R\to D$ and $\partial D$, and $u$ is not an internal cell of a boundary arc of $D$, then there is a good path from $u$ to a $0$-cell $v$ in the intersection of $\partial R\to D$ and $P_Y\to D$.
\end{enumerate}
\end{lem}
\begin{proof}
Suppose that all $S_i\to D$ factor through $P_L\to D$. By Lemma~\ref{lem:at-most-one-l-component},  without loss of generality,  we can assume that $S_i\to D\to X$ does not intersect $\ell$ for $i\geq 2$.
Then the path $P=T_1S_2\cdots S_nT_n \to D$ is a good path  since if a $0$-cell is mapped into $\ell$ then it is in the interior of $D$.  Since $u$ and $v$ are not internal cells of a boundary arc of $\partial D$,  $u$ and $v$ are in $P\to D$. The first statement follows.

Suppose $\partial R \to D$ intersects $P_L\to D$ but some $S_i\to D$ does not factor through $P_L\to D$.
Since $u$ is not an internal cell of a boundary arc, it follows that $u$ is an endpoint of $S_j \to D$ for some $j$. If $S_i\to D$ does not factor through $P_L\to D$, then it factors through $P_Y\to D$. Therefore $\partial R\to D$ has $0$-cells of $P_Y\to D$, and in particular there is a $0$-cell $v$ which is an endpoint of $S_i\to D$ for some $i$ and is in $P_Y\to D$ .
Lemma~\ref{lem:at-most-one-l-component} implies that at most one of the paths $S_i\to D\to X$ intersects $\ell$ and factors through $P_L\to D$. Without loss of generality, assume that if there is such a path then it is $S_1 \to D$. It follows that  $T_1S_2\cdots S_nT_n$ is a good path containing   $u$ and $v$.
\end{proof}

\begin{proof}[Proof of Proposition~\ref{prop:qcmain1}]
There is a sequence of cells $c_1, c_2, \ldots c_n$ in $D$ such that:
\begin{enumerate}
\item each $c_i$ is a open $2$-cell, or an open $1$-cell disconnecting $D$,
\item $\bar{c}_i \cap \bar{c}_{i+1}$ is either a cut $0$-cell or contains a $1$-cell, 
\item $u\in \bar{c}_1$ and $\bar{c}_n$ intersects $P_Y\to D$, 
\item $c_i$ is not equal to $c_{i+1}$,
\item\label{pp4} $\bar{c}_i$ does not intersect $P_Y\to D$ for $i<n$.
\end{enumerate}
Indeed, consider a path $S\to D$ from $u$ to  $P_Y\to D$. 
Each open $1$-cell of $S\to D$ is either 
disconnects $D$ or lies in the closure of $2$-cell. For consecutive edges $e_1,e_2$ that lie in $2$-cells either the $0$-cell $a$ that lies between $e_1$ and $e_2$ is a cut $0$-cell of $D$ or we add a sequence of $2$-cells corresponding to a path in $\link(a, D)$ between the vertices associated to $e_1$ and $e_2$. The last two properties are guaranteed by possibly passing to a subsequence.

If $n=1$ and $c_1$ is a $2$-cell, the proof concludes by letting $P\to D$ be a good path from $u$ to $P_Y\to D$ provided by Lemma~\ref{lem:basic-good-path}\eqref{lem:basic(2)}. If $n=1$ and $c_1$ is a $1$-cell, then the conclusion is immediate.

Suppose $n>1$. A good-path $Q\to D$ from $u$ to $P_Y\to D$ is constructed as a concatenation $Q=P_1\cdots P_n$ of good paths $P_i\to D$  as follows.    For each $i>0$, either let $b_i$ be a $1$-cell in $\bar c_i \cap \bar c_{i+1}$ or let $b_i$ be the cut $0$-cell of $D$ equal to $\bar c_i \cap \bar c_{i+1}$. Let $b_0=u$.  If $c_i$ is an isolated $1$-cell of $D$, then let $P_i\to D$ equal $c_i$. 

Suppose that $c_i$ is a $2$-cell and $i<n$.  Observe that if $b_i$ is a $0$-cell then $b_i$ is a cut $0$-cell of $D$ and, by Lemma~\ref{lem:cut-cells-2}, does not map into $\ell$. Suppose that $b_i$ is a $1$-cell. By Lemma~\ref{lem:at-most-one-l-component} it is impossible for both endpoints of $b_i$ to lie in $\partial D$ and lie on $P_L\to D$ and map into $\ell$.  If both endpoints of $b_i$ lie on $\partial D$, then property~\eqref{pp4} of the sequence $\{c_n\}$ implies that the endpoints of $b_i$ lie on $P_L\to Y$. Therefore either $b_i$ is a cut $0$-cell of $D$, or some (chosen) endpoint of $b_i$ is either internal or  does not map into $\ell$. 
Since $\bar c_i$ is disjoint from $P_Y\to D$,  Lemma~\ref{lem:basic-good-path}\eqref{lem:basic(1)} provides a good path $P_i\to \partial D$ from (our chosen points in) $b_{i-1}$ to $b_{i}$.  
Note that the hypotheses are satisfied. 
Indeed,  if $b_i$ is a $0$-cell then it is a cut $0$-cell of $D$ and therefore is not an internal $0$-cell of a boundary arc of $D$. 
If $b_i$ is a $1$-cell, then either the chosen endpoint is internal, or the chosen endpoint is on $\partial D$ and does not map into $\ell$ and is not an internal $0$-cell of a boundary arc of $D$. 

If $c_n$ is a $2$-cell, Lemma~\ref{lem:basic-good-path}\eqref{lem:basic(2)} implies that there is a good path $P_i\to \partial D$ from the chosen endpoint of $b_{n-1}$ to a $0$-cell of $P_Y\to D$.

Finally the good path $P\to D$ is the concatenation  $P_1\cdots P_n$ which is a good path by Lemma~\ref{lem:concatenation-goodpaths}.
\end{proof}

\subsection{Subdivisions, good paths, and internal paths} \label{subsec:subdivision-goodpath}

\begin{lem} [Subdividing along $\ell$]\label{lem:good-subdivision}
There is an $H$-equivariant subdivision $X'$ of $X$  satisfying:
\begin{enumerate}
\item $\ell$ is contained in the $1$-skeleton of $X'$,
\item each cell of $X'$ is convex,
\item each $0$-cell of $X'$ with nonzero curvature is a $0$-cell of $X$, and
\item if $X$ has nonpositive sectional curvature, then so does $X'$.
\end{enumerate}
\end{lem}
\begin{proof}
Since $\ell$ intersects finitely many cells and $H$ acts properly on $X$, we see that each cell intersects finitely many $H$-translates of $\ell$.
The one-skeleton of $X'$ equals $X^{1} \cup H \ell$. In particular, the zero-skeleton of $X'$ consists of three types of cells: $0$-cells of $X$, intersections of $H$-translates of $\ell$ with open $1$-cells of $X$, and self-intersections of distinct $H$-translates $\ell$ within open $2$-cells of $X$. Since each cell of $X$ is convex and $\ell$ is a geodesic segment, it follows that each cell of $X'$ is convex.

We now verify the third statement. Observe that each new $0$-cell in $X'$ is in the interior of either a $2$-cell or a $1$-cell of $X$. In the former case, the link is a circle with $2\pi$-angle sum.  In the latter case, the link is a finite subdivision of a bipartite graph $\Theta$, with two vertices and  each edge of $\Theta$ has angle $\pi$.

Finally, $\link (x, X')$ is a subdivision of  $\link (x, X)$ when $x \in X^0$. Thus the last statement follows.
\end{proof}

The subdivision $X'$ of $X$ induces a subdivision of any complex that is immersed in $X$. Induced subdivisions of $P_L$, $P_Y$, $D$, $Y$ and $L$ are denoted by of $P_L'$, $P_Y'$, $D'$, $Y'$ and $L'$ respectively. The geodesic $\ell$ is an edge path in $L'$. Divide the path $P_L' \to L'$ into  paths $P_1', \ldots , P_k'$ such that each $P_i'$ is  either a subpath of $\ell$ or else intersects $\ell$ only at its endpoints.
Let $\ell_i$  be the subpath of $\ell$ between the endpoints of $P_i'$. (Note that the concatenation $\ell_1 \ldots \ell_k$ may not equal $\ell$, although it is a path within $\ell$.)
Let $K_i' \subseteq L'$ be the subdiagram between $\ell_i$ and $P_i'$.  Let $K'$ be the complex obtained by attaching to $P_L'$ a copy of $K_i'$ along the subpath $P_i'$ for each $i$. Observe that $K' \to L'$ is a near-immersion, and  $P_L' \to K'$ and $\ell_1\ldots \ell_k \to K'$ are embeddings. Since $K'$ is contractible, $K'\to L'$ is a disk-diagram between $P_L'\to L'$ and $\ell_1 \ldots \ell_k \to L'$.   See Figure~\ref{fig:K-ladder}.

Let $E'$ be the complex obtained by identifying $D'$ and $K'$ along the images of $P_L' \to D'$ and $P_L' \to K'$. Since $D'$ and $K'$ are disk-diagrams and $P_L' \to K'$ is an embedding, it follows that $E'$ is a disk-diagram. The minimality assumption of $D$ implies that $E'\to X'$ is a near-immersion; see Lemma~\ref{lem:E-near-immersion} below.  Therefore $E'\to X'$ is a disk-diagram between $P_Y'\to X'$ and $\ell_1\ldots \ell_k \to X'$. See Figure~\ref{fig:E-Diagram}.

\begin{figure}\centering
\includegraphics[width=.5\textwidth]{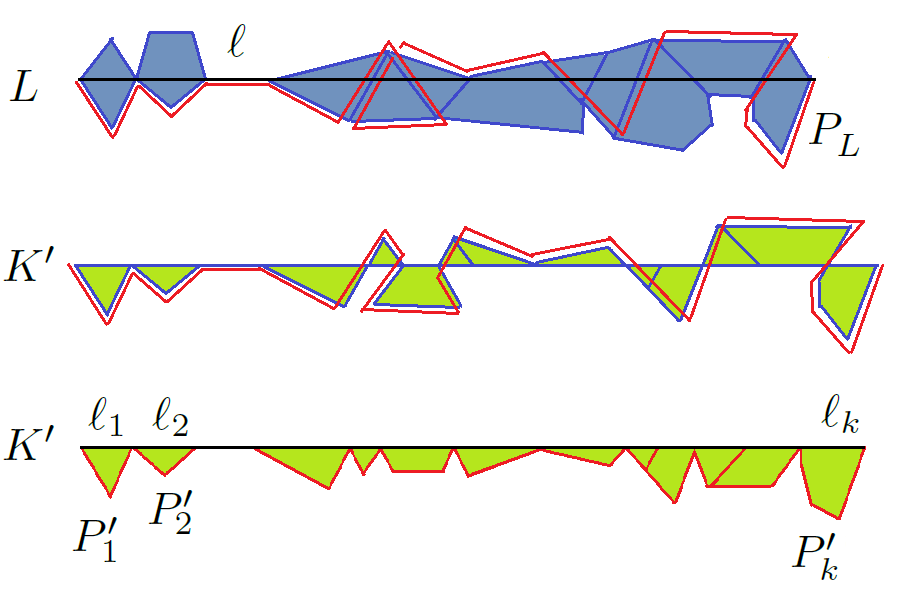}
\caption{The paths $\ell \to L$ and $P_L \to L$ are at the top. The embedding $P_L' \to K'$ is at the bottom.} \label{fig:K-ladder}
\end{figure}

\begin{lem}\label{lem:E-near-immersion}
The map $E' \to X'$ is a near-immersion.
\end{lem}
\begin{proof}
As $D\to X$ and $K'\to X'$ are near-immersions, it suffices to examine the $1$-cells of $E'$ along $P_L'\to E'$. Let $e'$ be a $1$-cell of $P_L'$. Suppose $R_1'$ and $R_2'$ are distinct  $2$-cells of $E'$ at (the image of) $e'$ that map to the same $2$-cell in $X'$. Since both $P_L' \to K'$  and $D'\to E'$ are embeddings, without loss of generality, assume $R_1'$ is in $K'$ and $R_2'$ is in $D'$.

We will show that the minimality of $D$ is violated.
Let $e$ be the $1$-cell of $P_L$ containing $e'$, let $R_1$ be the $2$-cell of $L$ containing the image of $R_1'$, and let $R_2$ be the $2$-cell of $D$ containing the image of $R_2'$. Let $\partial R_1 = Qe^{-1}$ and let $\dot{P}_L$ be formed from $P_L$ replacing $e$ with $Q$. Similarly, let $\dot{D}$ be the subdiagram of $D$ obtained by removing the open cells $e$ and $R_2$. Observe that  $\dot{D} \to X$ is a disk-diagram between  $\dot{P}_L \to X$ and $P_Y \to X$ and violates the minimality of $D$.

We now consider the case that both $R_1'$ and $R_2'$ are in $K'$. Observe that $R_1', R_2'$ meet each other along the $1$-cell $c'$ of $D'$ (the image of $e'$). Again, let $e$ be the $1$-cell of $P_L$ containing $e'$, and let $c$ be the $1$-cell of $D$ containing $c'$.  Then $P_L$ is a concatenation $S_1eS_2fS_3$ where $e$ and $f$ travel through the $1$-cell $c$. Let $\dot{P}_L$ equal $S_1S_3$ and note that $P_Y^{-1}\dot P_L \to D$ bounds a subdiagram $\dot D$ of $D$ - the part of $D$ that remains after removing the subdiagram bounded by $eS_2f$. Observe that $\dot P_L$ is a path in $L$ since $e$ and $f$ map to the same $1$-cell of $L$. As before $\dot{D}$ violates the minimality of $D$.
\end{proof}

\begin{figure}\centering
\includegraphics[width=.6\textwidth]{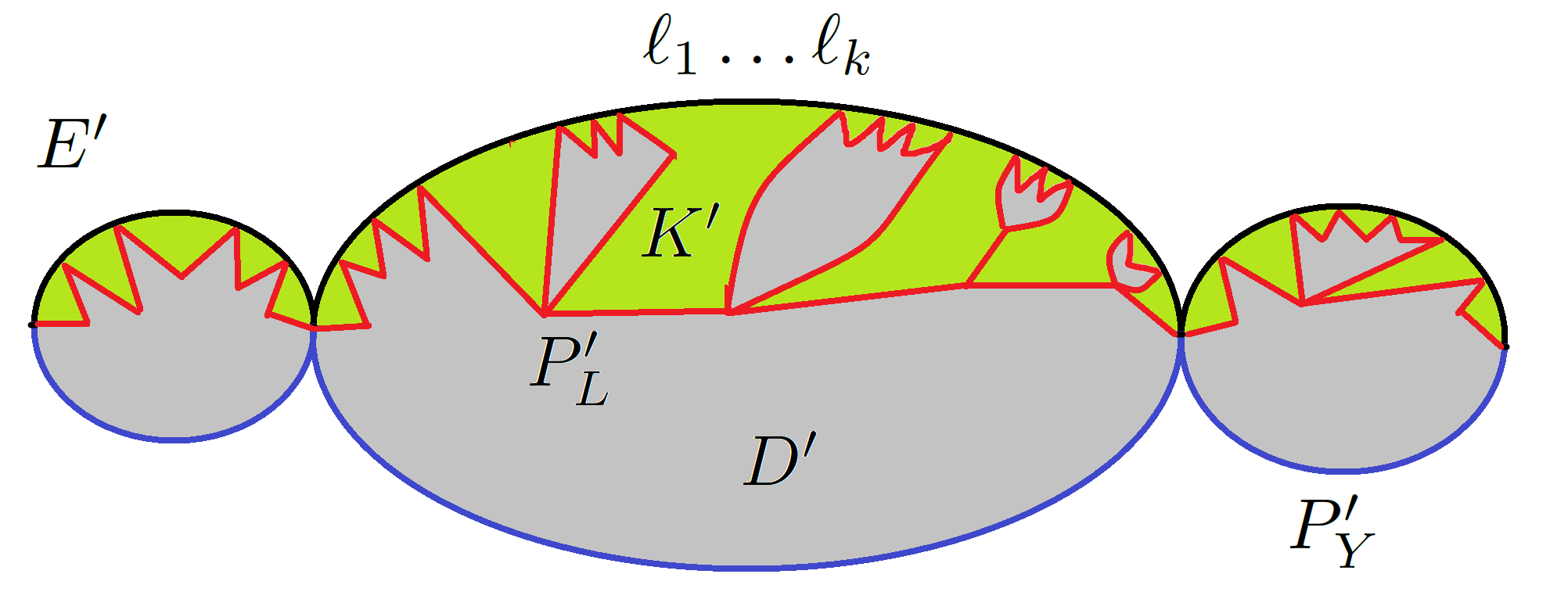}
\caption{The disk-diagram $E' \to X'$ between the paths $P_Y'\to X'$ and $\ell_1 \ldots \ell_k \to X'$.} \label{fig:E-Diagram}
\end{figure}

\subsection{The immersion $Z'\to X'$}

Let  $W' = Y' \sqcup_{H P_Y} H E'$ denote the union of $Y'$ and copies of $E'$ attached along the distinct $H$-translates of $P_Y$.  Since $E'$ is a finite complex, $W'$ is a cocompact $H$-complex. Since $Y'$ and $E'$ are simply connected, $W'$ is simply connected. In view of Lemma~\ref{lem:folding}, the $H$-map $W' \to X'$ factors as the composition of a surjection $W' \to Z'$ and an immersion $Z' \to X'$ where $Z'$ is a simply connected cocompact $H$-complex.

\begin{lem}[$\ell^2$-Euler Characteristic]
$\chi (H, Z')= \chi (H, Y)$.
\end{lem}
\begin{proof}
Since $Z'\to X'$ is a immersion, $Z'$ is a locally $CAT(0)$-space. Since $Z'$ is simply-connected, it is contractible. Analogously, $Y'$ is contractible.  Since $Y'$ and $Z'$ are both proper $H$-complexes and the embedding $Y' \to Z'$ is an $H$-map, Corollary~\ref{cor:betti-goal-2} implies that $\chi (H, Z')= \chi (H, Y')$. Moreover, $\chi (H, Y')=\chi (H, Y)$  since the definition of $\ell^2$-Betti numbers is independent of the cell structure of the space.
\end{proof}

\begin{lem}[$0$-cells with Positive Curvature]
$|\positivecells (H, Z')| \leq |\positivecells (H, Y)|$.
\end{lem}
\begin{proof}
It is enough to show that if $z$ is a $0$-cell of $Z'$ such that $\standardcurvature (\link (z), H_z)>0$  then $z$ is in the image of $Y'\to Z'$. Indeed, this statement implies that there is an injective map $\positivecells (H, Z') \to \positivecells (H, Y')$; and moreover $|\positivecells (H, Y')|=|\positivecells (H, Y)|$ since $0$-cells of $Y'$ arising as a result of the subdivision have zero curvature by Lemma~\ref{lem:good-subdivision}.

Let $z$ be a $0$-cell of $Z'$ and suppose that $z$ is not in the image of $Y'\to Z'$. We show that $\standardcurvature (\link (z), H_z)\leq 0$. Since $Z'$ has non-positive sectional curvature and it is positively angled, if $\link (z, Z')$ has a cycle or is disconnected then it is easy to see that $\standardcurvature (\link (z), H_z) \leq 0$.  Suppose that $\link (z, Z')$ is a tree.

Let $w$ be a preimage of $z$ by $W'\to Z'$. Since $z$ is not in the image of $Y'\to Z'$, it follows that $w$ is a $0$-cell in the image of $hE' \to W'$ for some $h\in H$. Without loss of generality, assume that $h=1$. Since $E'\to Z'\to X'$ is a near-immersion and $\link (z)$ is a tree, it follows that $\link (w)$ is a tree.  Since $E'$ is a disk-diagram, it follows that $w$ is in the image of $\partial E \to W'$.  Since the image of $w$ by $W'\to Z'$  is not contained in the image of $Y'\to Z'$, it follows that $w$ is not in the image of $P_Y'\to E'$. Therefore $w$ is a $0$-cell in the interior of $\ell_1 \ldots \ell_k \to W'$, and hence $z$ is in the interior of $\ell_1 \ldots \ell_k \to Z'$.

Suppose that $\ell_1 \ldots \ell_k \to Z'$ is locally a geodesic at $z$. By the construction of $K'$, if $e_1$ and $e_2$ are $1$-cells of $\ell_1 \ldots \ell_k$ with a common endpoint $z$ then $\link (z, Z')$ has a path between the vertices induced by $e_1$ and $e_2$ with angled-sum at least $\pi$, hence $\standardcurvature (\link (z), H_z) \leq 0$.

Suppose $\ell_1 \ldots \ell_k \to Z'$ is not locally an embedding around $z$ -there is a backtrack. By the construction of $K'$, $z$ is the  terminal point of $\ell_i$ for some $i<k$ and  the path $\ell_i \ell_{i+1} \to X'$ has a backtrack. It follows that $\link (z, K')$ consists of two non-edgeless components, and the angle sum of $\link (z, K')$ is $\pi$. Observe that the two components of $\link (z, K')$ are mapped into $\link (z, X')$ to a path with angle sum equal $\pi$; in particular, the angle sum of $\link (z, Z')$ is no less than $\pi$ and hence $\standardcurvature (\link (z), H_z)\leq 0$.
\end{proof}

\begin{lem}[$0$-cells with Negative Curvature]\label{lem:qcbound}
\begin{equation*}
\begin{split}
|\nega (H, Z')| &\leq   \frac{2\pi \cdot \chi (H, Y) - \positivebound (G, X) \cdot |\positivecells (H, Y)| }{ \negativebound (G, X)}. \qedhere
\end{split}
\end{equation*}
\end{lem}
\begin{proof}
Since angles are positive, observe that the constants of Definition~\ref{def:Bounds} satisty  $\negativebound(G, X') = \negativebound(G, X)$ and $\positivebound(G, X) = \positivebound(G, X')$. By Lemma~\ref{lem:Bounds},
\begin{equation*}
\begin{split}
|\nega (H, Z')| &\leq   \frac{2\pi \cdot \chi (H, Z') - \positivebound (G, X) \cdot |\positivecells (H, Z')| }{ \negativebound (G, X)}.
\end{split}
\end{equation*}
The conclusion follows from the previous two lemmas and the above inequality.
\end{proof}

\subsection{Conclusion of the proof of the Quasiconvex core theorem}

\begin{lem}[From good to internal]\label{lem:existence-goodpaths}
Suppose that $Q\to D$ is a good path whose interior does not intersect $P_Y\to D$. Then all $0$-cells of  $Q\to D$ are mapped to internal $0$-cells of $Z'$ by $Q'\to Z'$.
\end{lem}
\begin{proof}
Since $D'\to E'$ is an embedding each $0$-cell of the interior of $D'$ is mapped to the interior of $E'$.  Suppose $u$ is a $0$-cell of $Q\to D$ which is not in the interior of $D$.  Since $Q\to D$ is   good path, the image of $u$ in $X$ does not intersect $\ell$; and by assumption, $u$ is not in $P_Y\to D$. Therefore $u$ is  mapped into the interior of $E'$.  Since $E' \to Z'$ is a near-immersion the conclusion follows.
\end{proof}

Since $X$ is $G$-cocompact, there is an upper bound $C=C(X)$ on the length of boundary paths of $2$-cells of $X$.

\begin{lem}[$P_L$ is uniformly close to $P_Y$]\label{lem:PLPY}
Let $u$ be a $0$-cell of $P_L$. Then the combinatorial distance from $u$ to the subcomplex $Y$ is bounded by the constant
\[ 1 + C(X)+\frac{2\pi \cdot \chi (H, Y) - \positivebound (G, X) \cdot |\positivecells (H, Y)| }{ \negativebound (G, X)}.\]
\end{lem}
\begin{proof}
If $u$ is in the $C(X)$-neighborhood of $Y$ in $X$ then the statement is clear.  Otherwise, $u$ is not in $P_Y\to D$ and hence Lemma~\ref{lem:cut-cells-2} implies that $u$ is not a cut $0$-cell of $D$. It follows that  $u$ is in the closure of a boundary arc of $D$ and hence there is $v$ in $P_L$ which is not an internal cell of a boundary arc of $D$ and the distance between $u$ and $v$ is bounded by $C(X)$.

By Proposition~\ref{prop:qcmain1}, there is a good path $Q\to D$ from $v$ to a $0$-cell of $P_Y\to X$. 
Assume that $Q\to X$ has minimal combinatorial length.  Then the interior of $Q\to X$ does not intersect $P_Y\to D$.  By Lemma~\ref{lem:existence-goodpaths},  $Q\to X$ factors as $Q'\to Z'\to X'$ and each $0$-cell $z$ of $Q\to D$ is mapped to an internal cell of $Z'$.  

Since $Z'\to X'$ is an immersion, $\link(z, Z')$ has a cycle for each $0$-cell $z$ of $Q\to D$. 
Since $X$ has sectional curvature $\leq \alpha<0$, it follows that  $\standardcurvature (H_z, \link(z, Z'))<0$. By minimality, $Q\to X$ is embedded and, moreover, no pair of distinct $0$-cells of $Q'\to Z'$ are in the same $H$-orbit. Therefore $|Q|\leq 1+|\nega (H, Z')|$.

An upper bound for the combinatorial distance between $v$ and $P_Y\to X$  follows from the previous inequality and Lemma~\ref{lem:qcbound}. The proof concludes by adding the upper bound $C(X)$ on the distance between $u$ and $v$. 
\end{proof}

By Lemmas~\ref{lem:PLell}~and~\ref{lem:PLPY}, there is a uniform upper bound for the distance between $\ell$  and $Y$ independent of $\ell$. This concludes the proof of Theorem~\ref{thm:qccore}

\section{Large Quotients}

\begin{thm}\label{thm:quotients}
Let $X$ be a $CAT(0)$ cocompact and proper $G$-complex with sectional curvature $\leq \alpha<0$. Let $g \in G$ be an infinite order element. Let $\gamma$ be an axis for $g$ and suppose $\gamma \cap f \gamma$ is discrete for any $f \in G - \langle g \rangle$.  

There exists $N>0$ such that for any  $n\geq N$,  the group $\bar G = G / \nclose{g^n}$ has a CAT(0) cocompact $\bar G$-complex with sectional curvature $\leq \bar \alpha <0$.
\end{thm}

\begin{lem}\label{lem:adding-edge}
Let $\bar \Gamma$ be a graph, let $e$ be an edge, and let $\Gamma$ equal $\bar \Gamma - e$. Suppose that $\Gamma$ is a angled-graph with non-negative angles and sectional curvature $\leq \alpha \leq 0$, and suppose the angle distance $\measuredangle (\iota e, \tau e )$ in $\Gamma$ is $\geq \pi$.  Then $\bar \Gamma$ has sectional curvature $\leq 0$ by assigning $\measuredangle (e) =\pi$.
Suppose that $\Gamma$ has sectional curvature $\leq \alpha <0$,  $\measuredangle (\iota e, \tau e ) =\theta> \pi$, and $\measuredangle (e)>\pi+\alpha$. Then $\bar \Gamma$ has sectional curvature $< 0$.
\end{lem}
\begin{proof}
Let $\bar \Lambda$ be a connected, spurless and not edgeless subgraph of $\bar \Gamma$, and let $\Lambda$ be $\bar \Lambda \cap \Gamma$.  If $\Lambda = \bar \Lambda$ then $\standardcurvature (\bar \Lambda) =   \standardcurvature (\Lambda) \leq \alpha$ by hypothesis.   If $e$ is an edge of $\bar \Lambda$ then $\standardcurvature (\bar \Lambda) =    \standardcurvature (\Lambda) + \pi - \measuredangle (e)$.  Therefore  $\standardcurvature (\bar \Lambda)\leq \alpha + \pi - \measuredangle (e) \leq 0$ if $\Lambda$ was not a tree because removing spurs gives a section. Observe that the last inequality is strict if $\measuredangle (e)>\pi+\alpha$. Otherwise removing some spurs gives rise to a subdivided interval and hence
$\standardcurvature (\Lambda) \leq \pi - \measuredangle (\iota e, \tau e )$. In this case, $\standardcurvature (\bar \Lambda)\leq 0$ with strict inequality if $\measuredangle (\iota e, \tau e )>\pi$.
\end{proof}

\begin{proof}[Proof of Theorem~\ref{thm:quotients}]
Slightly decreasing all the angles yields a  $CAT(-\epsilon)$ structure on $X$. Let $\widehat X$ be the quotient of $X$ by $\nclose {g^n}$. 

Let $X'$ be a $G$-invariant subdivision of $X$ such that $\gamma$ lies in the $1$-skeleton of $X'$. By slightly increasing the curvature of all $2$-cells (increasing all the angles), $X'$ has negative sectional curvature at all new $0$-cells corresponding to intersections of $\gamma$ with $1$-cells of $X$ and translates of $\gamma$.

Let $\sigma_n$ be the subpath of $\gamma$ from $p$ to $g^np$. Form $\bar X$ from $\widehat X$ by attaching a $2$-cell $\bar h R$ along the cycle $\bar h \sigma^n$ of $\widehat X$;  attach a $2$-cell for each left coset  $\bar h \langle \bar g \rangle$ in $\bar G$. Extend the $\bar G$-action on $\hat X$ to $\bar X$ by letting each $\langle \bar g^{\bar h}$ act with a fixed point at the center of $\bar h R$.  Regard $R$ as an Euclidean $n$-gon whose $i$-th side is the translate of $\sigma_1$ by $\bar g^i$. Now we claim that for sufficiently large $n$, the complex $\bar X$ has negative sectional curvature.

By the assumption on $p$,  the links of vertices of $\bar X$ are independent of $n$.  By making all angles of $X$ slightly larger, we can assume that the angle distance between the initial and terminal vertices of corners along $\ell$ are $>\pi$.  By Lemma~\ref{lem:adding-edge}, for sufficiently large $n$, we can assign an angle of $(n-2)\pi/n$ to the corners of $R$ and its translates.  We assign an angle of $\pi$ at the other corners along the interior of $\sigma_1$. 

By subdividing $R$ into $n$ $2$-cells using its barycenter, and making its corners slightly smaller, we can assume that the barycenter has negative curvature and that $G$ acts without inversions on $\bar X$.
\end{proof}

\bibliographystyle{plain}
\bibliography{xbib}

\end{document}